\documentclass[12pt,a4paper]{article}
\usepackage{latexsym, amsmath, amsfonts, amscd, amssymb, verbatim, mathrsfs}
\def\N{I\!\!N}
\def\R{I\!\!R}
\def\oR{I\!\!\overline{R}}

\def\mN{\mathcal{N}}
\def\mC{\mathcal{C}}
\def\mE{\mathcal{E}}
\def\mP{\mathcal{P}}
\def\mJ{\mathcal{J}}

\def\mU{\mathcal{U}}

\def\be{\bar{e}}\def\we{\widetilde{e}}

\def\ou{\bar{u}}\def\wu{\widetilde{u}}\def\hu{\widehat{u}}
\def\ov{\bar{v}}

\def\oB{\bar{B}}

\def\cl{{\rm cl}}
\def\gph{{\rm gph\,}}
\def\epi{{\rm epi\,}}

\def\dom{{\rm dom\,}}

\def\cone{{\rm cone}}
\def\argmin{{\rm argmin}}

\def\st{|\;}\def\bst{\big|\;}\def\Bst{\Big|\;}
\def\Limsup{\mathop{{\rm Lim\hspace{0.3mm}sup}}}

\def\liminf{\mathop{{\rm lim\hspace{0.3mm}inf}}}

\begin{document}

\newtheorem{Theorem}{Theorem}[section]
\newtheorem{Proposition}[Theorem]{Proposition}
\newtheorem{Remark}[Theorem]{Remark}
\newtheorem{Lemma}[Theorem]{Lemma}
\newtheorem{Corollary}[Theorem]{Corollary}
\newtheorem{Definition}[Theorem]{Definition}
\newtheorem{Example}[Theorem]{Example}
\newtheorem{CounterExample}[Theorem]{Counter-Example}
\renewcommand{\theequation}{\thesection.\arabic{equation}}
\normalsize

\title{Full Stability for a Class of Control Problems of Semilinear
       Elliptic Partial Differential Equations}

\author{Nguyen Thanh Qui\footnote{College of Information and Communication Technology,
        Can Tho University, Campus II, 3/2 Street, Can Tho, Vietnam;
        ntqui@cit.ctu.edu.vn. Current address (01.07.2016-30.06.2018): Institut f\"{u}r Mathematik,
        Universit\"{a}t W\"{u}rzburg, Emil-Fischer-Str.~30, 97074 W\"{u}rzburg, Germany;
        thanhqui.nguyen@mathematik.uni-wuerzburg.de. The research of this author is
       supported by the Alexander von Humboldt-Foundation.}\quad and\quad
        Daniel Wachsmuth\footnote{Institut f\"{u}r Mathematik, Universit\"{a}t W\"{u}rzburg,
        Emil-Fischer-Str.~30, 97074 W\"{u}rzburg, Germany;
        daniel.wachsmuth@mathematik.uni-wuerzburg.de.}}

\maketitle
\date{}

\noindent {\bf Abstract.} We investigate full Lipschitzian and full
H\"{o}lderian stability for a class of control problems governed by
semilinear elliptic partial differential equations, where all the
cost functional, the state equation, and the admissible control set
of the control problems undergo perturbations. We establish explicit
characterizations of both Lipschitzian and H\"{o}lderian full
stability for the class of control problems. We show that for this
class of control problems the two full stability properties are
equivalent. In particular, the two properties are always equivalent
in general when the admissible control set is an arbitrary fixed
nonempty, closed, and convex set.



\section{Introduction}

The notion of \emph{full Lipschitzian stability} was introduced and
studied by Levy, Poliquin, and Rockafellar \cite{LePoRo00SIOPT} in
the finite dimensional setting. This property then was investigated
in the infinite dimensional setting by Mordukhovich and Nghia
\cite{MorNgh14SIOPT}. Moreover, in \cite{MorNgh14SIOPT}, the authors
also introduced and studied the notion of \emph{full H\"{o}lderian
stability}. The full Lipschitzian and H\"{o}lderian stability is
defined on the basis of two types of parameters: \emph{Basic
parameters} and \emph{tilt ones}, where the last type is related to
the notion of tilt stability introduced by Poliquin and Rockafellar
\cite{PoRo98SIOPT}.

According to \cite[Example~4.4]{MorNgh14SIOPT} (see also
\cite{Robin82MP} for the original example), the two notions of full
Lipschitzian and H\"{o}lderian stability are really different, where
the example shows that the full H\"{o}lderian stability is strictly
weaker than the Lipschitzian one. However, it is interesting to know
which classes of optimization models the Lipschitzian and
H\"{o}lderian full stability properties are equivalent.

In the paper \cite{MorNgh14SIOPT}, the authors provided various
characterizations of both Lipschitzian and H\"{o}lderian full
stability in general optimization models. Then they applied these
results to characterize  (only) the full Lipschitzian stability
of optimal control problems governed by semilinear elliptic partial
differential equations, where the basic parameters appear only in
the cost functional while the state equation and the admissible
control set of the problems are fixed.

In this paper, we study both Lipschitzian and H\"{o}lderian full
stability for a broader class of control problems governed by
semilinear elliptic partial differential equations. Moreover, all
the cost functional, the state equation, and the admissible control
set of the control problems undergo perturbations of basic
parameters. We will provide explicit characterizations of both
Lipschitzian and H\"{o}lderian full stability for the class of
optimal control problems on the basis of the second-order
subdifferential characterizations of the latter stability properties
for parametric optimization problems obtained in
\cite{MorNgh14SIOPT}. In particular, we show that for this class of
optimal control problems the two full stability properties are
equivalent.

Related to our class of control problems with the cost functional
not involving the usual quadratic term for the control and the
admissible control set being fixed, we refer the reader to
\cite{QuiWch17} for the H\"{o}lderian stability of bang-bang optimal
controls in $L^1$ that is not deduced from our results in this
paper; see also \cite{Cas12SICON,CasDWchGWch17,PonWch16OPTIM,PonWch17} for more details on
second-order optimality conditions, solution stability, and
numerical methods for bang-bang controls.

It is worth mentioning another approach to the local Lipschitzian
stability of solutions to parametric optimal control problems for
nonlinear systems; see, e.g., \cite{DontMal99CVOC,MalTro99ZAA,MalTro00CC} and the references therein.
The main tools in stability analysis for such problems are
Robinson's implicit function theorem for generalized equations given
in \cite{Robin80MOR} and an extension of Robinson's theorem obtained
in \cite{Dont95KAP}. However, Robinson's implicit function theorem
does not provide us with any information on the gap between
sufficient conditions and necessary conditions of Lipschitzian
stability while the extension of Robinson's theorem allows to
establish necessary and sufficient conditions of Lipschitzian
stability provided the sufficiently strong dependence of data on the
parameters. For instance, in \cite{MalTro00CC} the authors
established a second-order sufficient condition for Lipschitzian
stability of solutions to elliptic optimal control problems under
nonlinear perturbations in general. The latter condition is also a
second-order necessary condition for Lipschitzian stability only if
the parameter of the problems is a linear perturbed parameter.
Namely, in the terminology of full stability, there is only tilt
parameter in the problems while the basic parameter disappears. In
contrast to this approach, our results obtained in this paper are
necessary and sufficient conditions for Lipschitzian and
H\"{o}lderian stability for both basic and tilt perturbations.

The rest of the paper is organized as follows. A class of optimal
control problems together with assumptions posed on the initial data
of these problems are stated in Section~2. The section also recalls
some notions and facts from variational analysis and basic results
for optimal control problems of semilinear elliptic partial
equations.

In Section~3, we show that the two properties of Lipschitzian and
H\"{o}lderian full stability for a class of optimization problems,
where the cost functional and the state equation of the problems are
perturbed while the admissible control set of the problems is fixed,
are equivalent. This result is applied to deduce that the two
properties of full stability are always equivalent for our class of
optimal control problems when the admissible control set does not
undergo perturbations. In addition, explicit characterizations of
the full stability properties for the class of optimal control
problems in this case are also provided.

Section~4 is devoted to investigate full stability for the class of
optimal control problems, where all the cost functional, the state
equation, and the admissible control set of the problems are
perturbed. We establish explicit characterizations of both
Lipschitzian and H\"{o}lderian full stability properties for the
class of optimal control problems. Interestingly, the two full
stability properties are also equivalent in this setting. Note that
the equivalence of the two properties of full stability is due to
special structures of perturbed admissible control sets, the
properties of full stability are not always equivalent in general
for perturbed admissible control sets in contrast to the setting
considered in the previous section.

Some concluding remarks and further investigations on the full
stability are provided in the last section.

\section{Problem statement and preliminaries}

Consider the optimal control problem
\begin{equation}\label{OptConPro}
\begin{cases}
    {\rm Minimize}\quad J(u)=\displaystyle\int_\Omega L(x,y_u(x))dx
                            +\frac{1}{2}\displaystyle\int_\Omega \zeta(x)u(x)^2dx\\
    {\rm subject\ to}\quad\alpha(x)\leq u(x)\leq\beta(x)\quad\mbox{for a.e.}\ x\in\Omega,
\end{cases}
\end{equation}
where $\zeta\in L^2(\Omega)$ satisfies $\zeta(x)\geq\zeta_0>0$ for
a.e. $x\in\Omega$, and $y_u$ is the weak solution associated to the
control $u$ of the Dirichlet problem
\begin{equation}\label{StateEq}
\begin{cases}
\begin{aligned}
    Ay+f(x,y)&=u\ &&\mbox{in}\ \Omega\\
            y&=0  &&\mbox{on}\ \Gamma,
\end{aligned}
\end{cases}
\end{equation}
where $A$ denotes the second-order differential elliptic operator of
the form
\begin{equation}\label{SOdrEOper}
    Ay(x)=-\sum_{i,j=1}^N\partial_{x_j}\big(a_{ij}(x)\partial_{x_i}y(x)\big).
\end{equation}
From now on, we denote the set of admissible controls by
\begin{equation}\label{AdmCtrlSet}
    \mU_{ad}=\big\{u\in L^2(\Omega)\bst\alpha(x)\leq u(x)\leq\beta(x)~\mbox{for a.e.}~x\in\Omega\big\}.
\end{equation}
Observe that if $\alpha,\beta\in L^\infty(\Omega)$ with
$\alpha(x)\leq\beta(x)$ for a.e. $x\in\Omega$, then $\mU_{ad}$ is
nonempty, closed, bounded, and convex in $L^2(\Omega)$.

In this paper, we will study solution stability of
problem~\eqref{OptConPro}, where the state equation \eqref{StateEq},
the admissible control set \eqref{AdmCtrlSet}, and the cost
functional $J(\cdot)$ given in \eqref{OptConPro} undergo
perturbations. We are interested in the perturbed control problem as
follows
\begin{equation}\label{PerProWtCd}
\begin{cases}
    {\rm Minimize}\quad\mJ(u,e)=J(u+e_y)+(e_J,y_{u+e_y})_{L^2(\Omega)}\\
    {\rm subject~to}\quad u\in\mU_{ad}(e),
\end{cases}
\end{equation}
where $y_{u+e_y}$ is the weak solution associated to $u+e_y$ of the
perturbed Dirichlet problem
\begin{equation}\label{PerStaEqWtCd}
\begin{cases}
\begin{aligned}
    Ay+f(x,y)&=u+e_y\ &&\mbox{in}\ \Omega\\
            y&=0      &&\mbox{on}\ \Gamma,
\end{aligned}
\end{cases}
\end{equation}
the perturbed admissible control set $\mU_{ad}(e)$ is given by
\begin{equation}\label{PrAdCtrlSet}
    \mU_{ad}(e)=\big\{u\in L^2(\Omega)\bst\alpha(x)+e_\alpha(x)\leq u(x)
    \leq\beta(x)+e_\beta(x)~\mbox{for a.e.}~x\in\Omega\big\},
\end{equation}
and $e_J,e_y\in L^2(\Omega)$, $e_\alpha,e_\beta\in L^\infty(\Omega)$
are parameters. We denote the space of parameters by
$E=L^2(\Omega)\times L^2(\Omega)\times L^\infty(\Omega)\times
L^\infty(\Omega)$ with the norm $\|\cdot\|_E$ of
$e=(e_y,e_J,e_\alpha,e_\beta)\in E$ defined by
$$\|e\|_E=\|e_y\|_{L^2(\Omega)}+\|e_J\|_{L^2(\Omega)}+\|e_\alpha\|_{L^\infty(\Omega)}+\|e_\beta\|_{L^\infty(\Omega)}.$$
Let us fix any parameter $\be=(\be_J,\be_y,\be_\alpha,\be_\beta)\in
E$ and consider the following problem
\begin{equation}\label{FxturbPro}
    {\rm Minimize}\quad\mJ(u,\be)=J(u+\be_y)+(\be_J,y_{u+\be_y})_{L^2(\Omega)}\quad
    \mbox{subject to}\quad u\in\mU_{ad}(\be),
\end{equation}
where $y_{u+\be_y}$ is the weak solution associated to $u+\be_y$ of
the Dirichlet problem \eqref{PerStaEqWtCd}. The corresponding
perturbed problem of problem~\eqref{FxturbPro} is defined as follows
\begin{equation}\label{TprturbPro}
    \mP(u^*,e):\quad
    {\rm Minimize}\quad \mJ(u,e)-\langle u^*,u\rangle\quad
    \mbox{subject to}\quad u\in\mU_{ad}(e),
\end{equation}
where $\mJ:L^2(\Omega)\times E\to\R$ is defined in
\eqref{PerProWtCd} and $(u^*,e)\in L^2(\Omega)\times E$ are the
parameters. We interpret $e\in E$ as the basic parameter
perturbations and $u^*\in L^2(\Omega)$ as the tilt ones. Note that
problem~\eqref{FxturbPro} reduces to problem~\eqref{OptConPro} when
$\be=(0,0,0,0)\in E$.

The main contributions of this paper are the following:
\begin{itemize}
\item[(a)] We show that the two properties of full Lipschitzian and full
H\"{o}lderian stability for a class of optimization problems, where
the objective function is of class $\mC^2$ and the constraint set of
the problems is fixed, are always equivalent in general. We will
apply this result to deduce that the full Lipschitzian and full
H\"{o}lderian stability properties for problem~\eqref{TprturbPro}
are equivalent when the admissible control set $\mU_{ad}$ of the
control problem does not undergo perturbation, i.e., the basic
parameter is always in the form $e=(e_y,e_J,0,0)$. Moreover, we
provide explicit characterizations of the two properties for
problem~\eqref{TprturbPro} in this setting.

\item[(b)] We establish explicit characterizations of both Lipschitzian and
H\"{o}lderian full stability for problem~\eqref{TprturbPro}, where
all the cost functional, the state equation, and the admissible
control set of the problem undergo perturbations. In comparison
between the characterizations of the full stability properties
obtained we show that the two full stability properties are also
equivalent for this setting.

\item[(c)] Our results show that the equivalence of the Lipschitzian and H\"{o}lderian full
stability properties depend on the structure of perturbed admissible
control sets. Namely, the full stability properties are not always
equivalent when the admissible control sets undergo perturbations.
\end{itemize}

Given $(\ou^*,\be)\in L^2(\Omega)\times E$, $\ou\in\mU_{ad}(\be)$,
$(u^*,e)\in L^2(\Omega)\times E$, and $\gamma>0$, associated with
these data we define
\begin{equation}\label{ObjsStab}
\begin{cases}
    m_\gamma(u^*,e)=\displaystyle\inf_{u\in\mU_{ad}(e),\, \|u-\ou\|_{L^2(\Omega)}\leq\gamma}\big\{\mJ(u,e)-\langle u^*,u\rangle\big\},\\
    M_\gamma(u^*,e)=\underset{u\in\mU_{ad}(e),\, \|u-\ou\|_{L^2(\Omega)}\leq\gamma}\argmin\big\{\mJ(u,e)-\langle u^*,u\rangle\big\}.
\end{cases}
\end{equation}
Following \cite{LePoRo00SIOPT} and \cite{MorNgh14SIOPT}, we recall
the concepts of Lipschitzian and H\"{o}lderian full stability of the
problem $\mP(\ou^*,\be)$ in \eqref{TprturbPro} as follows:
\begin{itemize}
\item The control $\ou$ is said to be a \emph{Lipschitzian fully stable local
minimizer} of the problem $\mP(\ou^*,\be)$ if there exists a number
$\gamma>0$ such that the mapping $(u^*,e)\mapsto M_\gamma(u^*,e)$ in
\eqref{ObjsStab} is single-valued and locally Lipschitz continuous
with $M_\gamma(\ou^*,\be)=\ou$ and the function $(u^*,e)\mapsto
m_\gamma(u^*,e)$ is also Lipschitz continuous around $(\ou^*,\be)$.

\item We say that the control $\ou$ is a \emph{H\"{o}lderian fully stable local
minimizer} of the problem $\mP(\ou^*,\be)$ if there are
$\gamma,\kappa>0$ such that the mapping $(u^*,e)\mapsto
M_\gamma(u^*,e)$ from \eqref{ObjsStab} is single-valued around
$(\ou^*,\be)$ with $M_\gamma(\ou^*,\be)=\ou$ and the H\"{o}lder
property
\begin{equation}\label{HldrPrpty}
    \big\|M_\gamma(u^*,e)-M_\gamma(\wu^*,\we)\big\|_{L^2(\Omega)}
    \leq\kappa\Big(\|u^*-\wu^*\|_{L^2(\Omega)}+\|e-\we\|^{1/2}_{L^2(\Omega)}\Big)
\end{equation}
holds for any pairs $(u^*,e)$, $(\wu^*,\we)$ in a neighborhood
$U^*\times V$ of $(\ou^*,\be)$, and that the function
$(u^*,e)\mapsto m_\gamma(u^*,e)$ is Lipschitz continuous on
$U^*\times V$.
\end{itemize}

We now assume that $\Omega\subset\R^N$ with $N\in\{1,2,3\}$ and
$\alpha,\beta\in L^2(\Omega)$ with $\alpha(x)<\beta(x)$ for a.e.
$x\in\Omega$. Moreover, the function $L:\Omega\times\R\times\R\to\R$
are Carath\'eodory functions of class $\mC^2$ with respect to the
second and third variables satisfying the following assumptions.

\textbf{(A1)} The function $f$ is of class $\mC^2$ with respect to
the second variable, and
$$f(\cdot,0)\in L^2(\Omega)\quad\mbox{and}\quad\dfrac{\partial f}{\partial y}(x,y)\geq0\quad
  \mbox{for a.e.}\ x\in\Omega,$$
and for all $M>0$ there exists a constant $C_{f,M}>0$ such that
$$\left|\dfrac{\partial f}{\partial y}(x,y)\right|+\left|\dfrac{\partial^2f}{\partial y^2}(x,y)\right|\leq C_{f,M},$$
and
$$\left|\dfrac{\partial^2f}{\partial y^2}(x,y_2)-\dfrac{\partial^2f}{\partial y^2}(x,y_1)\right|
  \leq C_{f,M}|y_2-y_1|,$$
for a.e. $x\in\Omega$ and $|y|,|y_1|,|y_2|\leq M$.

\textbf{(A2)} The function $L(\cdot,0)\in L^1(\Omega)$ and for all
$M>0$ there are a constant $C_{L,M}>0$ and a function $\psi_M\in
L^2(\Omega)$ such that
$$\left|\dfrac{\partial L}{\partial y}(x,y)\right|\leq\psi_M(x),\quad
  \left|\dfrac{\partial^2L}{\partial y^2}(x,y)\right|\leq C_{L,M},$$
and
$$\left|\dfrac{\partial^2L}{\partial y^2}(x,y_2)-\dfrac{\partial^2L}{\partial y^2}(x,y_1)\right|
  \leq C_{L,M}|y_2-y_1|,$$
for a.e. $x\in\Omega$ and $|y|,|y_1|,|y_2|\leq M$.

\textbf{(A3)} The set $\Omega$ is an open and bounded domain in
$\R^N$ with Lipschitz boundary $\Gamma$, the coefficients $a_{ij}\in
L^\infty(\Omega)$ of the second-order differential elliptic operator
$A$ defined by \eqref{SOdrEOper} satisfy the condition
$$\lambda_A\|\xi\|^2_{\R^N}\leq\sum_{i,j=1}^Na_{ij}(x)\xi_i\xi_j,\ \forall\xi\in\R^N,\ \mbox{for a.e.}\ x\in\Omega,$$
for some constant $\lambda_A>0$.

\begin{Theorem}\label{ThmExSoStEq}{\rm(See \cite[Theorem~2.1]{CaReTr08SIOPT})}
Suppose that {\rm\textbf{(A1)}} holds. Then, for every $u\in
L^2(\Omega)$, the state equation~\eqref{StateEq} has a unique
solution $y_u\in H^1_0(\Omega)\cap C(\bar\Omega)$. In addition,
there exists a constant $M_{\alpha,\beta}$ such that
\begin{equation}\label{EstSolEqSt}
    \|y_u\|_{H^1_0(\Omega)}+\|y_u\|_{C(\bar\Omega)}\leq M_{\alpha,\beta},\ \forall u\in\mU_{ad}.
\end{equation}
Furthermore, if $u_n\rightharpoonup u$ weakly in $L^2(\Omega)$, then
$y_{u_n}\to y_u$ strongly in $H^1_0(\Omega)\cap C(\bar\Omega)$.
\end{Theorem}

\begin{Theorem}\label{Thm24CRT}{\rm(See \cite[Theorem~2.4]{CaReTr08SIOPT})}
If assumption~{\rm\textbf{(A1)}} holds, then the control-to-state
mapping $G:L^2(\Omega)\to H^1_0(\Omega)\cap C(\bar\Omega)$, defined
by $G(u)=y_u$, is of class $\mC^2$. Moreover, for every $u,v\in
L^2(\Omega)$, $z_{u,v}=G'(u)v$ is the unique weak solution of
\begin{equation}\label{EqSolZuv}
\begin{cases}
  \begin{aligned}
     Az+\frac{\partial f}{\partial y}(x,y)z&=v\ &&\mbox{in}\ \Omega\\
                                          z&=0  &&\mbox{on}\ \Gamma.
  \end{aligned}
\end{cases}
\end{equation}
Finally, for every $v_1,v_2\in L^2(\Omega)$,
$z_{v_1v_2}=G''(u)(v_1,v_2)$ is the unique weak solution of
\begin{equation}\label{EqSolSeGvv}
\begin{cases}
  \begin{aligned}
     Az+\frac{\partial f}{\partial y}(x,y)z+\frac{\partial^2f}{\partial y^2}(x,y)z_{u,v_1}z_{u,v_2}
      &=0\ &&\mbox{in}\ \Omega\\
     z&=0  &&\mbox{on}\ \Gamma,
  \end{aligned}
\end{cases}
\end{equation}
where $y=G(u)$ and $z_{u,v_i}=G'(u)v_i$ for $i=1,2$.
\end{Theorem}

Related to the results on the solution of the state equation
\eqref{StateEq} we refer the reader to \cite[Chapter~4]{Trolt10B}
for more details. We introduce the space $Y=H^1_0(\Omega)\cap
C(\bar\Omega)$ endowed with the norm
$$\|y\|_Y=\|y\|_{H^1_0(\Omega)}+\|y\|_{L^\infty(\Omega)}.$$

\begin{Theorem}\label{ThmFSdDerJ}{\rm(See \cite[Theorem~2.6 and Remark~2.8]{CaReTr08SIOPT})}
Suppose that {\rm\textbf{(A1)}} and {\rm\textbf{(A2)}} hold. The
cost functional $J:L^2(\Omega)\to\R$ is of class $\mC^2$. Moreover,
for every $u,v,v_1,v_2\in L^2(\Omega)$, the first and second
derivatives of $J(\cdot)$ are given by
\begin{equation}\label{FDeriCost}
    J'(u)v=\int_\Omega(\zeta u+\varphi_u)vdx,
\end{equation}
and
\begin{equation}\label{SDeriCost}
    J''(u)(v_1,v_2)
    =\int_\Omega\bigg(\dfrac{\partial^2L}{\partial y^2}(x,y_u)z_{u,v_1}z_{u,v_2}
     +\zeta v_1v_2-\varphi_u\dfrac{\partial^2f}{\partial y^2}(x,y_u)z_{u,v_1}z_{u,v_2}\bigg)dx,
\end{equation}
where $y_u=G(u)$, $z_{u,v_i}=G'(u)v_i$ for $i=1,2$, and
$\varphi_u\in W^{2,p}(\Omega)$ is the adjoint state of $y_u$ defined
as the unique  weak solution of
$$\begin{cases}
\begin{aligned}
    A^*\varphi+\dfrac{\partial f}{\partial y}(x,y_u)\varphi
                     &=\dfrac{\partial L}{\partial y}(x,y_u)\ &&\mbox{in}\ \Omega\\
             \varphi &=0                                      &&\mbox{on}\ \Gamma
\end{aligned}
\end{cases}$$
with $A^*$ being the adjoint operator of $A$.
\end{Theorem}

For any $p\in[1,\infty]$, we denote $\oB^p_\varepsilon(\ou)$ the
closed ball in the space $L^p(\Omega)$ with the center at $\ou\in
L^p(\Omega)$ and the radius $\varepsilon>0$, i.e.,
$$\oB^p_\varepsilon(\ou)=\{v\in L^p(\Omega)\bst\|v-\ou\|_{L^p(\Omega)}\leq\varepsilon\}.$$
An element $\ou\in\mU_{ad}$ is said to be a
\emph{solution}/\emph{global minimum} of problem~\eqref{OptConPro}
if $J(\ou)\leq J(u)$ for all $u\in\mU_{ad}$. We will say that $\ou$
is a \emph{local solution}/\emph{local minimum} of
problem~\eqref{OptConPro} in the sense of $L^p(\Omega)$ if there
exists a closed ball $\oB^p_\varepsilon(\ou)$ such that $J(\ou)\leq
J(u)$ for all $u\in\mU_{ad}\cap\oB^p_\varepsilon(\ou)$. The local
solution $\ou$ is called \emph{strict} if $J(\ou)<J(u)$ holds for
all $u\in\mU_{ad}\cap\oB^p_\varepsilon(\ou)$ with $u\neq\ou$. Under
the above assumptions {\rm\textbf{(A1)}}-{\rm\textbf{(A3)}},
solutions of problem~\eqref{OptConPro} exist.

\begin{Theorem}{\rm(See \cite[Theorem~2.2]{CaReTr08SIOPT})}
For the assumptions~{\rm\textbf{(A1)}}-{\rm\textbf{(A3)}}, the
control problem \eqref{OptConPro} has at least one solution.
\end{Theorem}

Let us recall concepts and facts of variational analysis and
generalized differentiation taken from \cite{Mor06Ba}. Unless
otherwise stated, every reference norm in a product normed space is
the sum norm. Given a point $u$ in a Banach space $X$ and $\rho>0$,
we denote $B_\rho(u)$ the open ball of center $u$ and radius $\rho$
in $X$, and $\oB_\rho(u)$ is the corresponding closed ball. Let
$F:X\rightrightarrows W$ be a multifunction between Banach spaces.
The graph of $F$, denoted by $\gph F$, is the set $\{(u,v)\in
X\times W\st v\in F(u)\}$. We say that $F$ is locally closed around
the point $\bar\omega=(\ou,\ov)\in\gph F$ if $\gph F$ is locally
closed around $\bar\omega$, i.e., there exists a closed ball
$\oB_\rho(\bar\omega)$ such that $\oB_\rho(\bar\omega)\cap\gph F$ is
closed in $X\times W$. For a multifunction $\Phi:X\rightrightarrows
X^*$, the \emph{sequential Painlev\'e-Kuratowski upper limit} of
$\Phi$ as $u\to \ou$ is defined by
\begin{equation}\label{OtrLimit}
\begin{aligned}
    \displaystyle\Limsup_{u\to\ou}\Phi(u)=\Big\{
    &u^*\in X^*\Bst\mbox{there exist}~u_n\to\ou~\mbox{and}~u^*_n\stackrel{w^*}\rightharpoonup u^*~\mbox{with}\\
    &u^*_n\in\Phi(u_n)~\mbox{for every}~k\in\N=\{1,2,\dots\}\Big\}.
\end{aligned}
\end{equation}
Let $\phi:X\to\oR$ be a proper extended-real-valued function on an
\emph{Asplund} space $X$ (see \cite{Asp68AM} for more details on
Asplund spaces). Assume that $\phi$ is lower semicontinuous (lsc)
around $\ou$ from the domain $\dom\phi=\{u\in X\st\phi(u)<\infty\}$.
The \emph{regular subdifferential} of $\phi$ at $\ou\in\dom\phi$ is
\begin{equation}\label{RegSubdif}
    \widehat{\partial}\phi(\ou)=\bigg\{u^*\in X^*\Bst\liminf_{u\to\ou}
    \frac{\phi(u)-\phi(\ou)-\langle u^*,u-\ou\rangle}{\|u-\ou\|}\geq0\bigg\},
\end{equation}
while the \emph{limiting subdifferential} (known also as
\emph{Mordukhovich subdifferential}) of $\phi$ at $\ou$ is defined
via the sequential outer limit \eqref{OtrLimit} by
\begin{equation}\label{LimSubdif}
    \partial\phi(\ou)=\Limsup_{u\stackrel{\phi}\to\ou}\widehat{\partial}\phi(u),
\end{equation}
where the notation $u\stackrel{\phi}\to\ou$ means that $u\to\ou$
with $\phi(u)\to\phi(\ou)$.

Given a nonempty set $\Theta\subset X$ locally closed around
$\ou\in\Omega$, the \emph{regular and limiting normal cones} to
$\Theta$ at $\ou\in\Theta$ are respectively defined by
\begin{equation}\label{RgLmtgNrCn}
    \widehat{N}(\ou;\Theta)=\widehat{\partial}\delta(\ou;\Theta)
    \quad\mbox{and}\quad N(\ou;\Theta)=\partial\delta(\ou;\Theta),
\end{equation}
where $\delta(\cdot;\Theta)$ is the indicator function of $\Theta$
defined by $\delta(u;\Theta)=0$ for $u\in\Theta$ and
$\delta(u;\Theta)=\infty$ otherwise. The \emph{regular} and
\emph{Mordukhovich coderivatives} of the multifunction
$F:X\rightrightarrows W$ at the point $(\ou,\ov)\in\gph F$ are
respectively the multifunction
$\widehat{D}^*F(\ou,\ov):W^*\rightrightarrows X^*$ defined by
$$\widehat{D}^*F(\ou,\ov)(v^*)=\big\{u^*\in X^*\bst(u^*,-v^*)\in\widehat{N}\big((\ou,\ov);\gph F\big)\big\},~\forall
  v^*\in W^*,$$
and the multifunction $D^*F(\ou,\ov):W^*\rightrightarrows X^*$ given
by
$$D^*F(\ou,\ov)(v^*)=\big\{u^*\in X^*\bst(u^*,-v^*)\in N\big((\ou,\ov);\gph F\big)\big\},~\forall v^*\in W^*.$$
Given any $\ou^*\in\partial\phi(\ou)$, the \emph{combined
second-order subdifferential} of $\phi$ at $\ou$ relative to $\ou^*$
is the multifunction
$\breve{\partial}^2\phi(\ou,\ou^*):X^{**}\rightrightarrows X^*$ with
the values
\begin{equation}\label{CSOdrSbdif}
    \breve{\partial}^2\phi(\ou,\ou^*)(u)=(\widehat{D}^*\partial\phi)(\ou,\ou^*)(u),~\forall u\in X^{**}.
\end{equation}
Note that for $\phi\in\mC^2$ around $\ou$ with
$\ou^*=\nabla\phi(\ou)$ we have
$\breve{\partial}^2\phi(\ou,\ou^*)(u)=\{\nabla^2\phi(\ou)u\}$ for
all $u\in X^{**}$ via the symmetric Hessian operator
$\nabla^2\phi(\ou)$.

We say that the multifunction $F:X\rightrightarrows W$ is
\emph{locally Lipschitz-like}, or $F$ has the \emph{Aubin property}
\cite{DontRock09B}, around a point $(\ou,\ov)\in\gph F$ if there
exist $\ell>0$ and neighborhoods $U$ of $\ou$, $V$ of $\ov$ such
that
$$F(u_1)\cap V\subset F(u_2)+\ell\|u_1-u_2\|\oB_W,~\forall u_1,u_2\in U,$$
where $\oB_W$ denotes the closed unit ball in $W$. Characterization
of this property via the mixed Mordukhovich coderivative of $F$ can
be found in \cite[Theorem~4.10]{Mor06Ba}.

Let us recall the concepts of prox-regularity and subdifferential
continuity of extended-real-valued functions from
\cite{MorNgh14SIOPT}. Given a function $\psi:X\times E\to\oR$ finite
at $(\ou,\be)$ and given a partial limiting subgradient
$\ou^*\in\partial_u\psi(\ou,\be)$ of $\psi(\cdot,\be)$ at $\ou$, we
say that $\psi$ is \emph{prox-regular} in $u$ for $\ou^*$ with
\emph{compatible parameterization} by $e$ at $\be$ if there are
neighborhoods $U$ of $\ou$, $U^*$ of $\ou^*$, and $V$ of $\be$ along
with numbers $\varepsilon>0$ and $r>0$ such that
$$\psi(u,e)\geq\psi(v,e)+\langle u^*,u-v\rangle-\frac{r}{2}\|u-v\|^2,~\forall u\in U,$$
whenever
$$u^*\in\partial_u\psi(v,e)\cap U^*,~(v,e)\in U\times V,~\psi(v,e)\leq\psi(\ou,\be)+\varepsilon.$$
The function $\psi$ is \emph{subdifferentially continuous} in $u$ at
$\ou^*$ with \emph{compatible parameterization} by $e$ at $\be$ if
the mapping $(u,e,u^*)\mapsto\psi(u,e)$ is continuous relative to
$\gph\partial_u\psi$ at $(\ou,\be,\ou^*)$. When $\psi$ is
prox-regular and subdifferentially continuous in $u$ at $\ou^*$ with
compatible parameterization by $e$ at $\be$, $\psi$ is said to be
\emph{parametrically continuously prox-regular} at $(\ou,\be)$ for
$\ou^*$. These concepts are comprehensively studied in finite
dimensional settings; see \cite{LePoRo00SIOPT}. The nonparametric
versions of these concepts can be found in \cite{PoRo96TAMS},
\cite{RoWe98B}. We say that the \emph{basic constraint
qualification} (BCQ) holds at a point $(\ou,\be)\in\dom\psi$ if the
epigraphical mapping $F:E\rightrightarrows X\times\R$ defined by
$F(e)=\epi\psi(\cdot,e)$ is locally Lipschitz-like around
$(\be,\ou,\psi(\ou,\be))\in\gph F$, where the set
$$\epi\psi(\cdot,e):=\{(u,\tau)\in X\times\R\st\tau\geq\psi(u,e)\}$$
is the epigraph of the function $\psi(\cdot,e)$.

\section{Fixed admissible control set and full stability}

Let $X$ be a Hilbert space and let $E$ be an Asplund space. Fix any
$\be\in E$ and consider the problem
\begin{equation}\label{FxHatPro}
    {\rm Minimize}\quad \psi(u,\be)\quad\mbox{over}\quad u\in X,
\end{equation}
where $\psi:X\times E\to\oR$ is an extended-real-valued function.
The corresponding perturbed problem of \eqref{FxHatPro} is as
follows
\begin{equation}\label{PtrbHatPro}
    \mP_{Abs}(u^*,e):\quad{\rm Minimize}\quad \psi(u,e)-\langle u^*,u\rangle\quad\mbox{over}\quad u\in X.
\end{equation}
Here, the subscript ``Abs" refers to the abstract nature of
this optimization problem. We recall
\cite[Theorem~4.9]{MorNgh14SIOPT} on characterization for full
Lipschitzian stability via a perturbed positive definiteness of the
regular coderivative of the function $\partial_u\psi(\cdot,\cdot)$.

\begin{Theorem}\label{ThmLipChRg}{\rm(See \cite[Theorem~4.9]{MorNgh14SIOPT})}
Assume that the BCQ holds at $(\ou,\be)\in\dom\psi$ and that $\psi$
is parametrically continuously prox-regular at $(\ou,\be)$ for
$\ou^*\in\partial_u\psi(\ou,\be)$. Then, the following are
equivalent:
\begin{itemize}
\item[{\rm(i)}] The point $\ou$ is a Lipschitzian fully stable local
minimizer of $\mP_{Abs}(\ou^*,\be)$ in \eqref{PtrbHatPro}.
\item[{\rm(ii)}] The graphical mapping $e\mapsto\gph\partial_u\psi(\cdot,e)$ is locally Lipschitz-like
around $(\be,\ou,\ou^*)$ and there are $\eta,\delta>0$ such that for
all $(u,e,u^*)\in\gph\partial_u\psi\cap\oB_\eta(\ou,\be,\ou^*)$ we
have
\begin{equation}\label{LipChrCB}
    \langle v^*,v\rangle\geq\delta\|v\|^2_{L^2(\Omega)}~\;\mbox{whenever}~\;
    (v^*,e^*)\in\widehat{D}^*(\partial_u\psi)(u,e,u^*)(v)~\mbox{for}~v\in X.
\end{equation}
\end{itemize}
\end{Theorem}

\noindent And, we also recall \cite[Theorem~4.7]{MorNgh14SIOPT} on
characterization for full H\"{o}lderian stability via a perturbed
positive definiteness of the combined second-order subdifferential
of the function $\psi_e(\cdot)=\psi(\cdot,e)$.

\begin{Theorem}\label{ThmHdrCmBn}{\rm(See \cite[Theorem~4.7]{MorNgh14SIOPT})}
Assume that the BCQ holds at $(\ou,\be)\in\dom\psi$ and that $\psi$
is parametrically continuously prox-regular at $(\ou,\be)$ for
$\ou^*\in\partial_u\psi(\ou,\be)$. Then, the following are
equivalent:
\begin{itemize}
\item[{\rm(i)}] The point $\ou$ is a H\"{o}lderian fully stable local
minimizer of $\mP_{Abs}(\ou^*,\be)$ in \eqref{PtrbHatPro}.
\item[{\rm(ii)}] There are $\eta,\delta>0$ such that for all
$(u,e,u^*)\in\gph\partial_u\psi\cap\oB_\eta(\ou,\be,\ou^*)$ we have
\begin{equation}\label{HldChrCB}
    \langle v^*,v\rangle\geq\delta\|v\|^2_{L^2(\Omega)}~\;\mbox{whenever}~\;
    v^*\in\breve{\partial}^2\psi_e(u,u^*)(v)~\mbox{for}~v\in X.
\end{equation}
\end{itemize}
\end{Theorem}

We now consider the case that
$$\psi(u,e)=\phi(u,e)+\delta(u;K),$$
where $\phi:X\times E\to\oR$ is $\mC^2$ around the point
$(\ou,\be)\in\dom\psi$ and $K$ is a closed and convex subset of $X$.
This means that problem~\eqref{FxHatPro} can be rewritten as follows
$${\rm Minimize}\quad\phi(u,\be)\quad\mbox{subject to}\quad u\in K.$$
In this case, we prove that the properties of full Lipschitzian
stability and full H\"{o}lderian stability are equivalent.

\begin{Theorem}\label{ThmLpHdrEqv}
Let $(\ou,\be)\in\dom\psi$ and $\ou^*\in\partial_u\psi(\ou,\be)$ be
given. The following statements are equivalent:
\begin{itemize}
\item[{\rm(i)}] The point $\ou$ is a Lipschitzian fully stable local
minimizer of $\mP_{Abs}(\ou^*,\be)$ in \eqref{PtrbHatPro}.
\item[{\rm(ii)}] The point $\ou$ is a H\"{o}lderian fully stable local
minimizer of $\mP_{Abs}(\ou^*,\be)$ in \eqref{PtrbHatPro}.
\item[{\rm(iii)}] There are $\eta,\delta>0$ such that for
all $(u,e,u^*)\in\gph\partial_u\psi\cap\oB_\eta(\ou,\be,\ou^*)$ we
have
$$\langle v^*,v\rangle\geq\delta\|v\|^2_{L^2(\Omega)}~\;\mbox{whenever}~\;
  (v^*,e^*)\in\widehat{D}^*(\partial_u\psi)(u,e,u^*)(v)~\mbox{for}~v\in X.$$
\item[{\rm(iv)}] There are $\eta,\delta>0$ such that for all
$(u,e,u^*)\in\gph\partial_u\psi\cap\oB_\eta(\ou,\be,\ou^*)$ we have
$$\langle v^*,v\rangle\geq\delta\|v\|^2_{L^2(\Omega)}~\;\mbox{whenever}~\;
  v^*\in\breve{\partial}^2\psi_e(u,u^*)(v)~\mbox{for}~v\in X.$$
\end{itemize}
\end{Theorem}
\begin{proof}
According to the proof of \cite[Theorem~6.3]{MorNgh14SIOPT}, the BCQ
holds at $(\ou,\be)\in\dom\psi$ and $\psi$ is parametrically
continuously prox-regular at $(\ou,\be)$ for
$\ou^*\in\partial_u\psi(\ou,\be)$. We now verify that the graphical
mapping $F(e):=\gph\partial_u\psi(\cdot,e)$ is locally
Lipschitz-like around $(\be,\ou,\ou^*)$. Take any
$e_1,e_2\in\oB_\delta(\be)$ and $(u,u^*)\in
F(e_1)\cap\oB_\delta(\be,\ou,\ou^*)$ for $\delta$ small enough.
Then, by setting $\wu^*:=u^*+\phi'_u(u,e_2)-\phi'_u(u,e_1)$, we have
$(u,\wu^*)\in F(e_2)$. Hence, we obtain
$$(u,u^*)\in(u,\wu^*)+\ell\|e_1-e_2\|_E\oB_{L^2(\Omega)\times L^2(\Omega)},$$
where $\ell>0$ is a Lipschitz constant of $\phi'_u(\cdot,\cdot)$
around $(\ou,\be)$. This shows that
$$F(e_1)\cap V\subset F(e_2)+\ell\|e_1-e_2\|_E\oB_{L^2(\Omega)\times E},~\forall e_1,e_2\in U,$$
where $U:=\oB_\delta(\be)$ and $V:=\oB_\delta(\be,\ou,\ou^*)$. This
means that $F$ is locally Lipschitz-like around
$(\be,\ou,\ou^*)\in\gph F$.

In addition, for each
$(u,e,u^*)\in\gph\partial_u\psi\cap\oB_\eta(\ou,\be,\ou^*)$ and for
every $v\in X$, it holds from \cite[Theorem~1.62]{Mor06Ba} that
$$\begin{aligned}
  \widehat{D}^*(\partial_u\psi)(u,e,u^*)(v)
  &=\widehat{D}^*\big(\phi'_u(\cdot,\cdot)+N(\cdot;K)\big)(u,e,u^*)(v)\\
  &=\big(\phi''_{uu}(u,e)v,\phi''_{ue}(u,e)v\big)+\widehat{D}^*N(\cdot;K)\big(u,u^*-\phi'_u(u,e)\big)(v)\times\{0_E\}\\
  &=\Big(\phi''_{uu}(u,e)v+\widehat{D}^*N(\cdot;K)\big(u,u^*-\phi'_u(u,e)\big)(v)\Big)\times\{\phi''_{ue}(u,e)v\}\\
  &=\Big((\phi_e)''(u)v+\widehat{D}^*N(\cdot;K)\big(u,u^*-(\phi_e)'(u)\big)(v)\Big)\times\{\phi''_{ue}(u,e)v\}\\
  &=\widehat{D}^*\big((\phi_e)'(\cdot)+N(\cdot;K)\big)(u,u^*)(v)\times\{\phi''_{ue}(u,e)v\}\\
  &=\breve{\partial}^2\psi_e(u,u^*)(v)\times\{\phi''_{ue}(u,e)v\}.
\end{aligned}$$
Hence, we obtain
$$\begin{aligned}
    (v^*,e^*)\in\widehat{D}^*(\partial_u\psi)(u,e,u^*)(v)
    &\Longleftrightarrow
     \begin{cases}
        v^*\in\breve{\partial}^2\psi_e(u,u^*)(v)\\
        e^*=\phi''_{ue}(u,e)v.
     \end{cases}
\end{aligned}$$
This implies that \eqref{LipChrCB} is equivalent to
\eqref{HldChrCB}. Applying Theorems~\ref{ThmLipChRg} and
\ref{ThmHdrCmBn} we get the assertion of the theorem. $\hfill\Box$
\end{proof}

\medskip
We are going to apply Theorem~\ref{ThmLpHdrEqv} to provide explicit
characterizations of full stability properties for
problem~\eqref{FxturbPro}, where the admissible control set of the
problem is fixed. In this case, problem~\eqref{FxturbPro} is
rewritten as follows
\begin{equation}\label{FxAdCtrPro}
    {\rm Minimize}\quad\mJ(u,\be)=J(u+\be_y)+(\be_J,y_{u+\be_y})_{L^2(\Omega)}\quad
    \mbox{subject to}\quad u\in\mU_{ad},
\end{equation}
and the corresponding perturbed problem of
problem~\eqref{FxAdCtrPro} is defined by
\begin{equation}\label{PrFxAdCtrPro}
    \mP_0(u^*,e):\quad
    {\rm Minimize}\quad \mJ(u,e)-\langle u^*,u\rangle\quad
    \mbox{subject to}\quad u\in\mU_{ad},
\end{equation}
where the basic parameter $e\in E$ always appears in the form
$e=(e_y,e_J,0,0)$.

Following \cite{Har77JMSJ}, we say that a closed and convex subset
$K$ of a Banach space $X$ is \emph{polyhedric} at $\ou\in K$ for
$\hu^*\in N(\ou;K)$ if we have the representation
\begin{equation}\label{CndPlyhrc}
    T_K(\ou)\cap\{\hu^*\}^\bot=\cl\big(\cone(K-\ou)\cap\{\hu^*\}^\bot\big),
\end{equation}
where $\cone(K-\ou)=\bigcup_{t>0}t^{-1}(K-\ou)$ is the radial cone
and $T_K(\ou)=\cl(\cone(K-\ou))$ is the tangent cone to $K$ at
$\ou$. The set $K$ is said to be \emph{polyhedric} if $K$ is
polyhedric at every $u\in K$ for any $u^*\in N(u;K)$. The
polyhedricity property of a set is first introduced in
\cite{Har77JMSJ} and then applied extensively in optimal control;
see, e.g., \cite{Bon98AMO}, \cite{BonSha00B}, \cite{ItoKun08B} and
the references therein.

\begin{Theorem}{\rm(See \cite[Theorem~6.2]{MorNgh14SIOPT})}
For any $\ou\in K$ and $\hu^*\in N(\ou;K)$, we have
\begin{equation}\label{DomCmBnSbd}
    \dom\breve{\partial}^2\delta(\cdot;K)(\ou,\hu^*)\subset-\big(T_K(\ou)\cap\{\hu^*\}^\bot\big).
\end{equation}
If, in addition, $K$ is polyhedric at $\ou\in K$ for $\hu^*$, then
the equality
\begin{equation}\label{EqCmBnSbdf}
    \breve{\partial}^2\delta(\cdot;K)(\ou,\hu^*)(u)=\big(T_K(\ou)\cap\{\hu^*\}^\bot\big)^*
\end{equation}
holds for all $u\in-\big(T_K(\ou)\cap\{\hu^*\}^\bot\big)$.
\end{Theorem}

\begin{Remark}\label{RmkUadPoly}\rm
According to \cite[Lemma~4.13]{BaBoSi14TAMS} (see also
\cite[Lemma~2.4]{KiNhSo17SVAA}), the set of admissible controls
$\mU_{ad}$ is polyhedric at every $u\in\mU_{ad}$ for any $u^*\in
N(u;\mU_{ad})$, and thus $\mU_{ad}$ is polyhedric. Note that this
property also holds for the perturbed admissible control set
$\mU_{ad}(e)$ in \eqref{PrAdCtrlSet} for $e\in E$. This result is a
particular instance of the more general result
\cite[Theorem~3.58]{BonSha00B} which holds true in general Banach
lattices.
\end{Remark}

From now on, for every pair $(u,u^*)$ with $u^*\in N(u;\mU_{ad})$,
we define the critical cone
\begin{equation}\label{CriCone}
    C_0(u,u^*)=T_{\mU_{ad}}(u)\cap\{u^*\}^\bot.
\end{equation}
Since $\mU_{ad}$ is polyhedric at every $u\in\mU_{ad}$ for any
$u^*\in N(u;\mU_{ad})$, from \eqref{EqCmBnSbdf} and \eqref{CriCone}
we deduce that
\begin{equation}\label{EqCBnSdfCne}
    \breve{\partial}^2\delta(\cdot;\mU_{ad})(u,u^*)(v)=C_0(u,u^*)^*,~\forall v\in-C_0(u,u^*).
\end{equation}
Let us define the \emph{normal cone mapping}
$\mN_0:L^2(\Omega)\rightrightarrows L^2(\Omega)$ by setting
\begin{equation}\label{NorMapng}
    \mN_0(u)=N(u;\mU_{ad}),\ \forall u\in L^2(\Omega).
\end{equation}
Applying Theorem~\ref{ThmLpHdrEqv} and
\cite[Theorem~6.3]{MorNgh14SIOPT}, we obtain the second-order
characterization of Lipschitzian and H\"{o}lderian full stability
for the problem $\mP_0(\ou^*,\be)$ in the following theorem.

\begin{Theorem}\label{ThmCFLPPbm}
Assume that the assumptions {\rm\textbf{(A1)}-\textbf{(A3)}} hold.
Given $(\ou,\be)\in\mU_{ad}\times E$ with $\be=(\be_y,\be_J,0,0)$,
let $\ou^*\in\mJ'_u(\ou,\be)+\mN_0(\ou)$ and define
$\hu^*=\ou^*-\mJ'_u(\ou,\be)\in\mN_0(\ou)$. Then, the following are
equivalent:
\begin{itemize}
\item[{\rm(i)}] The control $\ou$ is a Lipschitzian fully stable
local minimizer for $\mP_0(\ou^*,\be)$ in \eqref{PrFxAdCtrPro}.
\item[{\rm(ii)}] The control $\ou$ is a H\"{o}lderian fully stable
local minimizer for $\mP_0(\ou^*,\be)$ in \eqref{PrFxAdCtrPro}.
\item[{\rm(iii)}] There exist $\eta>0$ and $\delta>0$ such that for
$(u,u^*)\in\gph\mN_0\cap\oB_\eta(\ou,\hu^*)$ and
$e\in\oB_\eta(\be)$, we have
\begin{equation}\label{CondChrFLp}
    \mJ''_{uu}(u,e)v^2\geq\delta\|v\|^2_{L^2(\Omega)},~\forall v\in C_0(u,u^*),
\end{equation}
where $C_0(u,u^*)$ is defined by \eqref{CriCone}.
\end{itemize}
\end{Theorem}
\begin{proof}
We see that $(\ou,\be)\in\dom\psi$ and
$\ou^*\in\partial_u\psi(\ou,\be)$, where
$\psi(u,e)=\mJ(u,e)+\delta(u;\mU_{ad})$ with $\mJ(\cdot,\cdot)$
being $\mC^2$ around $(\ou,\be)$. Since the set $\mU_{ad}$ is
polyhedric by Remark~\ref{RmkUadPoly}, according to
\cite[Theorem~6.3]{MorNgh14SIOPT} the control $\ou$ is a
Lipschitzian fully stable local minimizer for $\mP_0(\ou^*,\be)$ in
\eqref{PrFxAdCtrPro} if and only if \eqref{CondChrFLp} holds.
Therefore, applying Theorem~\ref{ThmLpHdrEqv} we obtain the
assertion of the theorem. $\hfill\Box$
\end{proof}

\medskip
This theorem shows that a certain condition on the positive
definiteness of $\mJ''_{uu}$ near $(\ou,\be)$ is necessary and
sufficient for stability. In the remainder of this section, we will
derive an equivalent condition, which is posed only on
$\mJ''_{uu}(\ou,\be)$.

\medskip
We now analyze Theorem~\ref{ThmCFLPPbm} to derive an explicit
characterization for Lipschitzian and H\"{o}lderian full stability
of the problem $\mP_0(\ou^*,\be)$. We denote
$\mC^0_{w^*}(\ou,\hu^*)$ the sequential outer limit of critical
cones $C_0(u,u^*)$ in the weak* topology of $L^2(\Omega)$, i.e.,
\begin{equation}\label{OlmWStopo}
\begin{aligned}
    \mC^0_{w^*}(\ou,\hu^*)
    &=\Limsup_{(u,u^*)\stackrel{\gph\mN_0}\longrightarrow(\ou,\hu^*)}C_0(u,u^*)\\
    &=\Big\{v\in L^2(\Omega)\Bst\exists(u_n,u^*_n)
      \stackrel{\gph\mN_0}\longrightarrow(\ou,\hu^*),v_n\in C_0(u_n,u^*_n),v_n\stackrel{w^*}\rightharpoonup v\Big\},
\end{aligned}
\end{equation}
and $\mC^0_s(\ou,\hu^*)$ the sequential outer limit of $C_0(u,u^*)$
in the strong topology of $L^2(\Omega)$, i.e.,
\begin{equation}\label{OlmSTtopo}
    \mC^0_s(\ou,\hu^*)=\Big\{v\in L^2(\Omega)\Bst\exists(u_n,u^*_n)
    \stackrel{\gph\mN_0}\longrightarrow(\ou,\hu^*),v_n\in C_0(u_n,u^*_n),v_n\to v\Big\}.
\end{equation}

\begin{Lemma}\label{LemOlmWSST}
Let $\ou\in\mU_{ad}$ and $\hu^*\in\mN_0(\ou)$. Then, both
$\mC^0_{w^*}(\ou,\hu^*)$ and $\mC^0_s(\ou,\hu^*)$ are computed by
the formula
\begin{equation}\label{CompWSSTc}
    \mC^0_{w^*}(\ou,\hu^*)=\mC^0_s(\ou,\hu^*)
    =\big\{v\in L^2(\Omega)\bst v(x)\hu^*(x)=0\ \mbox{for a.e.}\ x\in\Omega\big\},
\end{equation}
where $\mC^0_{w^*}(\ou,\hu^*)$ and $\mC^0_s(\ou,\hu^*)$ are
respectively given by \eqref{OlmWStopo} and \eqref{OlmSTtopo}.
\end{Lemma}
\begin{proof}
The set of admissible controls $\mU_{ad}$ is convex and polyhedric
due to Remark~\ref{RmkUadPoly}. In addition, by
\cite[Lemma~4.11]{BaBoSi14TAMS}, we have the representation of
$T_{\mU_{ad}}(\ou)$ as follows
\begin{equation}\label{RepreTUad}
    T_{\mU_{ad}}(\ou)
    =\left\{v\in L^2(\Omega)\left|\,
    \begin{aligned}
        &v(x)\geq0\ \mbox{for}\ x\in\{\ou=\alpha\}\\
        &v(x)\leq0\ \mbox{for}\ x\in\{\ou=\beta\}
    \end{aligned}\right.\right\},
\end{equation}
where
$$\{\ou=\alpha\}=\{x\in\Omega\st \ou(x)=\alpha(x)\}\quad\mbox{and}\quad
  \{\ou=\beta\}=\{x\in\Omega\st \ou(x)=\beta(x)\}.$$
Using the convexity and polyhedricity of $\mU_{ad}$ and the
representation of $T_{\mU_{ad}}(\ou)$ in \eqref{RepreTUad}, by
arguing similarly as in the proof of
\cite[Proposition~7.3]{MorNgh14SIOPT} we obtain \eqref{CompWSSTc}.
$\hfill\Box$
\end{proof}

\medskip
A quadratic form $Q:H\to\R$ on a Hilbert space $H$ is said to be a
\emph{Legendre form} if $Q$ is sequentially weakly lower
semicontinuous and that if $h_n$ converges weakly to $h$ in $H$ and
$Q(h_n)\to Q(h)$ then $h_n$ converges strongly to $h$ in $H$.

\begin{Lemma}\label{LmQLgndrFm}
For any $(\ou,\be_y)\in\mU_{ad}\times L^2(\Omega)$, define the
quadratic form $Q:L^2(\Omega)\to\R$ by
\begin{equation}\label{QdrtcFrmQ}
    Q(h):=\mJ''_{uu}(\ou,\be)h^2,\ \forall h\in L^2(\Omega),
\end{equation}
where $\mJ(\cdot,\cdot)$ is given in \eqref{PerProWtCd}. Then, $Q$
is a Legendre form on $L^2(\Omega)$.
\end{Lemma}
\begin{proof}
We first show that the quadratic form $R(h):=J''(\ou+\be_y)h^2$
defined on $L^2(\Omega)$, where $J(\cdot)$ is given in
\eqref{OptConPro}, is a Legendre form on $L^2(\Omega)$. Due to
\eqref{SDeriCost}, we have
$$\begin{aligned}
    R(h)
    &=\int_\Omega\bigg(\dfrac{\partial^2L}{\partial y^2}(x,y_{\ou+\be_y})z^2_{{\ou+\be_y},h}
     +\zeta h^2-\varphi_{\ou+\be_y}\dfrac{\partial^2f}{\partial y^2}(x,y_{\ou+\be_y})z^2_{{\ou+\be_y},h}\bigg)dx\\
    &=R_1(h)+R_2(h),
\end{aligned}$$
where
$$R_1(h):=\int_\Omega\bigg(\dfrac{\partial^2L}{\partial y^2}(x,y_{\ou+\be_y})-\varphi_{\ou+\be_y}
  \dfrac{\partial^2f}{\partial y^2}(x,y_{\ou+\be_y})\bigg)z^2_{{\ou+\be_y},h}dx$$
and
$$R_2(h):=\int_\Omega\zeta h^2dx.$$
Since $G'(\ou+\be_y):h\mapsto z_{\ou+\be_y,h}$ from $L^2(\Omega)$
into $L^2(\Omega)$ is compact, $R_1(h)$ is a weakly continuous
quadratic form on $L^2(\Omega)$. In addition, we observe that
$R_1(th)=t^2R_1(h)$ for all $h\in L^2(\Omega)$ and $t>0$, i.e.,
$R_1$ is positively homogeneous of degree $2$. On the other hand,
$R_2$ is an elliptic quadratic form on $L^2(\Omega)$ because $R_2$
is continuous and
$$R_2(h)=\int_\Omega\zeta h^2dx\geq\zeta_0\|h\|_{L^2(\Omega)},~\forall h\in L^2(\Omega),$$
where $\zeta(x)\geq\zeta_0>0$ for a.e. $x\in\Omega$. Due to
\cite[Proposition~3.76]{BonSha00B}, $R_2$ is a Legendre form on
$L^2(\Omega)$, and thus $R=R_1+R_2$ is also a Legendre form on
$L^2(\Omega)$.

We now verify that $Q$ is a Legendre form on $L^2(\Omega)$. Indeed,
we see that $Q$ is sequentially weakly lower semicontinuous.
Moreover, we have
$$\begin{aligned}
    Q(h)&=\mJ''_{uu}(\ou,\be)h^2=J''(\ou+\be_y)h^2+\big(\be_J,G''(\ou+\be_y)h^2\big)_{L^2(\Omega)}\\
        &=R(h)+\big(\be_J,G''(\ou+\be_y)h^2\big)_{L^2(\Omega)},
\end{aligned}$$
where $R(h)=J''(\ou+\be_y)h^2$ is a Legendre form on $L^2(\Omega)$.
Suppose that $h_n$ converges weakly to $h$ in $L^2(\Omega)$ and
$Q(h_n)\to Q(h)$. Then, $z_{\ou+\be_y,h_n}$ converges strongly to
$z_{\ou+\be_y,h}$ in $L^2(\Omega)$, and thus $G''(\ou+\be_y)h^2_n$
converges strongly to $G''(\ou+\be_y)h^2$ in $L^2(\Omega)$.
Consequently, we have
$$\big(\be_J,G''(\ou+\be_y)h^2_n\big)_{L^2(\Omega)}\to\big(\be_J,G''(\ou+\be_y)h^2\big)_{L^2(\Omega)}.$$
Combining this with
$R(h_n)=Q(h_n)-\big(\be_J,G''(\ou+\be_y)h^2_n\big)_{L^2(\Omega)}$ we
deduce that
$$R(h_n)\to Q(h)-\big(\be_J,G''(\ou+\be_y)h^2\big)_{L^2(\Omega)}=R(h).$$
This implies that $h_n$ converges strongly to $h$ since $R$ is a
Legendre form. Therefore, we have shown that $Q$ is a Legendre form.
$\hfill\Box$
\end{proof}

\begin{Lemma}\label{LmJCntnous}
Consider $\mJ(\cdot,\cdot)$ in \eqref{PerProWtCd}. The map
$\mJ''_{uu}:L^2(\Omega)\times E\to L^2(\Omega)\times L^2(\Omega)$
defined by setting
\begin{equation}\label{}
    (u,e)\mapsto\mJ''_{uu}(u,e),\ \forall (u,e)\in L^2(\Omega)\times E,
\end{equation}
is continuous on $L^2(\Omega)\times E$.
\end{Lemma}
\begin{proof}
Let any sequence $(u,e)$ converge to $(\ou,\be)$ in
$L^2(\Omega)\times E$, where $e=(e_y,e_J,e_\alpha,e_\beta)$ and
$\be=(\be_y,\be_J,\be_\alpha,\be_\beta)$. According to
\eqref{PerProWtCd}, we have
$$\mJ(u,e)=J(u+e_y)+(e_J,y_{u+e_y})_{L^2(\Omega)}=J(u+e_y)+\big(e_J,G(u+e_y)\big)_{L^2(\Omega)}.$$
By Theorems~\ref{Thm24CRT} and \ref{ThmFSdDerJ}, $J(\cdot)$ and
$G(\cdot)$ are of class $\mC^2$, thus we have
$$\mJ''_{uu}(u,e)(h_1,h_2)=J''(u+e_y)(h_1,h_2)+\big(e_J,G''(u+e_y)(h_1,h_2)\big)_{L^2(\Omega)}.$$
In addition, when $(u,e)\to(\ou,\be)$, we have $u+e_y\to\ou+\be_y$
and $e_J\to\be_J$, and it follows that
\begin{equation}\label{JGestmate}
    J''(u+e_y)\to J''(\ou+\be_y)\quad\mbox{and}\quad G''(u+e_y)\to G''(\ou+\be_y)
\end{equation}
strongly in $L^2(\Omega)\times L^2(\Omega)$ and
$\mathscr{L}(L^2(\Omega)\times L^2(\Omega),L^2(\Omega))$,
respectively. Moreover, we have
\begin{equation}\label{eJGestmte}
\begin{aligned}
    &\big\|\big(e_J,G''(u+e_y)(\cdot,\cdot)\big)_{L^2(\Omega)}
        -\big(\be_J,G''(\ou+\be_y)(\cdot,\cdot)\big)_{L^2(\Omega)}\big\|_{L^2(\Omega)\times L^2(\Omega)}\\
    &\qquad\leq\big\|\big(e_J,G''(u+e_y)(\cdot,\cdot)
        -G''(\ou+\be_J)(\cdot,\cdot)\big)_{L^2(\Omega)}\big\|_{L^2(\Omega)\times L^2(\Omega)}\\
    &\qquad\quad+\big\|\big(e_J-\be_J,G''(\ou+\be_y)(\cdot,\cdot)\big)_{L^2(\Omega)}\big\|_{L^2(\Omega)
        \times L^2(\Omega)}\\
    &\qquad=\sup_{\|h_1\|=\|h_2\|=1}\Big|\big(e_J,[G''(u+e_y)
        -G''(\ou+\be_y)](h_1,h_2)\big)_{L^2(\Omega)}\Big|\\
    &\qquad\quad+\sup_{\|h_1\|=\|h_2\|=1}\Big|\big(e_J-\be_J,G''(\ou+\be_y)(h_1,h_2)\big)_{L^2(\Omega)}\Big|\\
    &\qquad\leq\|e_J\|_{L^2(\Omega)}\sup_{\|h_1\|=\|h_2\|=1}\|[G''(u+e_y)
        -G''(\ou+\be_y)](h_1,h_2)\|_{L^2(\Omega)}\\
    &\qquad\quad+\|e_J-\be_J\|_{L^2(\Omega)}\sup_{\|h_1\|=\|h_2\|=1}\|G''(\ou+\be_y)(h_1,h_2)\|_{L^2(\Omega)}\\
    &\qquad\leq\|e_J\|_{L^2(\Omega)}\|G''(u+e_y)-G''(\ou+\be_y)\|_{\mathscr{L}(L^2(\Omega)\times
        L^2(\Omega),L^2(\Omega))}\\
    &\qquad\quad+\|e_J-\be_J\|_{L^2(\Omega)}\|G''(\ou+\be_y)\|_{\mathscr{L}(L^2(\Omega)\times L^2(\Omega),
        L^2(\Omega))}\\
    &\qquad\to0.
\end{aligned}
\end{equation}
From \eqref{JGestmate} and \eqref{eJGestmte} we deduce that
$$\|\mJ''_{uu}(u,e)-\mJ''_{uu}(\ou,\be)\|_{L^2(\Omega)\times L^2(\Omega)}\to0.$$
This show that the map $(u,e)\mapsto\mJ''_{uu}(u,e)$ is continuous
on $L^2(\Omega)\times E$. $\hfill\Box$
\end{proof}

\begin{Theorem}\label{ThmChLpBrUE}
Assume that the assumptions {\rm\textbf{(A1)}-\textbf{(A3)}} hold.
Given $(\ou,\be)\in\mU_{ad}\times E$ with $\be=(\be_y,\be_J,0,0)$,
let
$$\ou^*\in\zeta(\ou+\be_y)+\varphi_{\ou+\be_y}+G'(\ou+\be_y)^*\be_J+\mN_0(\ou)$$
and define
$$\hu^*=\ou^*-\zeta(\ou+\be_y)-\varphi_{\ou+\be_y}-G'(\ou+\be_y)^*\be_J.$$
Then, the following statements are equivalent:
\begin{itemize}
\item[{\rm(i)}] The control $\ou$ is a Lipschitzian fully stable local
minimizer for $\mP_0(\ou^*,\be)$ in \eqref{PrFxAdCtrPro}.
\item[{\rm(ii)}] The control $\ou$ is a H\"{o}lderian fully stable local
minimizer for $\mP_0(\ou^*,\be)$ in \eqref{PrFxAdCtrPro}.
\item[{\rm(iii)}] The following condition holds that
\begin{equation}\label{ExCdChrFLp}
    \mJ''_{uu}(\ou,\be)v^2>0,\ \forall v\neq0\ \mbox{with}\ v(x)\hu^*(x)=0~\mbox{for a.e.}~x\in\Omega,
\end{equation}
where $\be=(\be_y,\be_J,0,0)$.
\end{itemize}
\end{Theorem}
\begin{proof}
According to \eqref{FDeriCost}, we have
\begin{equation}\label{DeriMathJ}
\begin{aligned}
    \mJ'_u(\ou,\be)v
    &=J'(\ou+\be_y)v+\big(\be_J,G'(\ou+\be_y)v\big)_{L^2(\Omega)}\\
    &=\int_\Omega\big(\zeta(\ou+\be_y)+\varphi_{\ou+\be_y}\big)vdx+\int_\Omega G'(\ou+\be_y)^*\be_Jvdx\\
    &=\int_\Omega\big(\zeta(\ou+\be_y)+\varphi_{\ou+\be_y}+G'(\ou+\be_y)^*\be_J\big)vdx.
\end{aligned}
\end{equation}
From \eqref{DeriMathJ} it follows that
$$\mJ'_u(\ou,\be)
  =\zeta(\ou+\be_y)+\varphi_{\ou+\be_y}+G'(\ou+\be_y)^*\be_J\in L^2(\Omega).$$
Thus, we obtain the inclusions $\ou^*\in\mJ'_u(\ou,\be)+\mN_0(\ou)$
and $\hu^*\in\mN_0(\ou)$.

By Theorem~\ref{ThmCFLPPbm}, (i) is equivalent to (ii).

We now verify that (i) is equivalent to (iii). Assume that the
control $\ou$ is a Lipschitzian fully stable local minimizer for the
problem $\mP_0(\ou^*,\be)$. By Theorem~\ref{ThmCFLPPbm}, there exist
$\eta>0$ and $\delta>0$ such that for each
$(u,u^*)\in\gph\mN_0\cap\oB_\eta(\ou,\hu^*)$ and
$e\in\oB_\eta(\be)$, we have \eqref{CondChrFLp}. Fix any $v\in
L^2(\Omega)$ with $v\neq0$ and $v(x)\hu^*(x)=0$ for a.e.
$x\in\Omega$. It follows that $v\in\mC^0_s(\ou,\hu^*)$ due to
\eqref{CompWSSTc}. From \eqref{OlmSTtopo} one can find sequences
$(u_n,u^*_n)\to(\ou,\hu^*)$ with $(u_n,u^*_n)\in\gph\mN_0$, $v_n\in
C_0(u_n,u^*_n)$ such that $v_n\to v$ as $n\to\infty$. Since $v_n\in
C_0(u_n,u^*_n)$, from \eqref{CondChrFLp} we deduce that
\begin{equation}\label{SqnCdChrFLp}
    \mJ''_{uu}(u_n,e_n)v^2_n\geq\delta\|v_n\|^2_{L^2(\Omega)},\ \forall n\in\N,
\end{equation}
where $e_n=\be$ for every $n\in\N$. By passing \eqref{SqnCdChrFLp}
to the limit as $n\to\infty$, we get \eqref{ExCdChrFLp}.

Conversely, suppose that \eqref{ExCdChrFLp} holds. In order to apply
Theorem~\ref{ThmCFLPPbm} to deduce that $\ou$ is a Lipschitzian
fully stable local minimizer for the problem $\mP_0(\ou^*,\be)$, we
have to prove that \eqref{CondChrFLp} holds. Suppose to the contrary
that \eqref{CondChrFLp} does not hold. Then, one can find sequences
$(u_n,u^*_n)\to(\ou,\hu^*)$ with $(u_n,u^*_n)\in\gph\mN_0$,
$e_n\to\be$, and $v_n\in C_0(u_n,u^*_n)$ such that
\begin{equation}\label{CtrryCond}
    \mJ''_{uu}(u_n,e_n)v^2_n<\frac{1}{n}\|v_n\|^2_{L^2(\Omega)},\ \forall n\in\N,
\end{equation}
where $e_n=(e_{yn},e_{Jn},0,0)$ for every $n\in\N$. We may assume
that $\|v_n\|_{L^2(\Omega)}=1$ for every $n\in\N$. Without loss of
generality we may assume that $v_n\rightharpoonup v$ weakly (and
also weakly star) in $L^2(\Omega)$. From \eqref{OlmWStopo} we deduce
that $v\in\mC^0_{w^*}(\ou,\hu^*)$, which implies $v(x)\hu^*(x)=0$
for a.e. $x\in\Omega$ due to Lemma~\ref{LemOlmWSST}. Now, on one
hand, we have
\begin{equation}\label{DermJVn}
\begin{aligned}
    \mJ''_{uu}(u_n,e_n)v^2_n
    &=J''(u_n+e_{yn})v^2_n+\big(e_{Jn},G''(u_n+e_{yn})v^2_n\big)_{L^2(\Omega)}\\
    &=\int_\Omega\bigg(\dfrac{\partial^2L}{\partial y^2}(x,y_{u_n+e_{yn}})-\varphi_{u_n+e_{yn}}
      \dfrac{\partial^2f}{\partial y^2}(x,y_{u_n+e_{yn}})\bigg)z^2_{{u_n+e_{yn}},v_n}dx\\
    &\quad +\int_\Omega\zeta v^2_ndx+\big(e_{Jn},G''(u_n+e_{yn})v^2_n\big)_{L^2(\Omega)}.
\end{aligned}
\end{equation}
By our assumptions posed on the functions $L$ and $f$, since
$R_2(h)=\int_\Omega\zeta h^2dx$ is sequentially weakly lover
semicontinuous and $G(\cdot)$ is of class $\mC^2$, from
\eqref{DermJVn} we deduce that
$$\begin{aligned}
    \liminf_{n\to\infty}\mJ''_{uu}(u_n,e_n)v^2_n
    &\geq\int_\Omega\bigg(\dfrac{\partial^2L}{\partial y^2}(x,y_{\ou+\be_y})-\varphi_{\ou+\be_y}
      \dfrac{\partial^2f}{\partial y^2}(x,y_{\ou+\be_y})\bigg)z^2_{{\ou+\be_y},v}dx\\
    &\quad +\int_\Omega\zeta v^2dx+\big(\be_J,G''(\ou+\be_y)v^2\big)_{L^2(\Omega)}\\
    &=\mJ''_{uu}(\ou,\be)v^2.
\end{aligned}$$
From this and \eqref{CtrryCond} we get
$\mJ''_{uu}(\ou,\be)v^2\leq0$. This yields $v=0$ due to
\eqref{ExCdChrFLp}. It follows that
\begin{equation}\label{EqLimitJ}
    \lim_{n\to\infty}\mJ''_{uu}(u_n,e_n)v^2_n=\mJ''_{uu}(\ou,\be)v^2=0.
\end{equation}
On the other hand, we have
\begin{equation}\label{SbtrScdDer}
\begin{aligned}
    &\mJ''_{uu}(u_n,e_n)v^2_n-\mJ''_{uu}(\ou,\be)v^2\\
    &\qquad =\mJ''_{uu}(u_n,e_n)v^2_n-\mJ''_{uu}(\ou,\be)v^2_n+\mJ''_{uu}(\ou,\be)v^2_n-\mJ''_{uu}(\ou,\be)v^2\\
    &\qquad =\big(\mJ''_{uu}(u_n,e_n)-\mJ''_{uu}(\ou,\be)\big)v^2_n+Q(v_n)-Q(v).
\end{aligned}
\end{equation}
where $Q$ is the Legendre form defined by \eqref{QdrtcFrmQ}. From
\eqref{EqLimitJ}, \eqref{SbtrScdDer} and by Lemma~\ref{LmJCntnous}
we infer that
$$Q(v_n)-Q(v)=\mJ''_{uu}(u_n,e_n)v^2_n-\mJ''_{uu}(\ou,\be)v^2
  -\big(\mJ''_{uu}(u_n,e_n)-\mJ''_{uu}(\ou,\be)\big)v^2_n\to0,$$
when $n\to\infty$. Since $Q$ is a Legendre form on $L^2(\Omega)$,
$v_n$ converges strongly to $v$ (with $v=0$) in $L^2(\Omega)$. We
arrive at a contradiction due to $\|v_n\|_{L^2(\Omega)}=1$ for every
$n\in\N$. $\hfill\Box$
\end{proof}

\begin{Remark}\label{RmkFxCSwMT}\rm
In \cite{MalTro00CC}, it is proved that a \emph{coercivity
condition} (called \textbf{(AC)} in \cite{MalTro00CC}) equivalent to
\eqref{ExCdChrFLp} is sufficient for the local Lipschitzian
stability of solutions to parametric optimal control problems for
nonlinear systems. In addition, this condition was proved to be
necessary for local Lipschitzian stability of solutions with respect
to only tilt parameter. The results in \cite{MalTro00CC} are
established on the basis of Robinson's implicit function theorem
\cite{Robin80MOR} and its extension obtained in \cite{Dont95KAP}.
See also \cite{MalTro99ZAA} for the local Lipschitzian stability of
solutions to parametric optimal control problems for parabolic
equations.
\end{Remark}

\section{Perturbed admissible control set and full stability}

In this section, we will establish explicit characterizations of the
full stability properties for problem~\eqref{FxturbPro} with the
corresponding perturbed problem \eqref{TprturbPro}, where the
perturbed admissible control set $\mU_{ad}(e)$ is given by
\eqref{PrAdCtrlSet} with respect to $e\in E=L^2(\Omega)\times
L^2(\Omega)\times L^\infty(\Omega)\times L^\infty(\Omega)$.

Let us consider the function $\psi:L^2(\Omega)\times E\to\R$ defined
by
\begin{equation}\label{FncPsi}
    \psi(u,e)=\mJ(u,e)+\delta(u;\mU_{ad}(e)),
\end{equation}
and the \emph{normal cone mapping} $\mN:L^2(\Omega)\times
E\rightrightarrows L^2(\Omega)$ given by
\begin{equation}\label{AdCtrNrMap}
    \mN(u,e)=N(u;\mU_{ad}(e)),
\end{equation}
for every $(u,e)\in L^2(\Omega)\times E$. For each $e\in E$, we put
$\mJ_e(u)=\mJ(u,e)$, $\delta_e(u)=\delta(u;\mU_{ad}(e))$,
$\mN_e(u)=\mN(u,e)$, and $\psi_e(u)=\mJ_e(u)+\delta_e(u)$ for all
$u\in L^2(\Omega)$.

For every $(u,e,u^*)$ with $u^*\in\mN(u,e)$, we define the critical
cone
\begin{equation}\label{AdCtrCrCone}
    C(u,e,u^*)=T_{\mU_{ad}(e)}(u)\cap\{u^*\}^\bot.
\end{equation}
For each $e\in E$, using \eqref{DomCmBnSbd} with $K=\mU_{ad}(e)$ we
obtain
$$\dom\breve{\partial}^2\delta_e(u,u^*)=\dom\breve{\partial}^2\delta_e(\cdot;\mU_{ad}(e))(u,u^*)
  \subset-\big(T_{\mU_{ad}(e)}(u)\cap\{u^*\}^\bot\big),$$
and thus combining this with \eqref{AdCtrCrCone} yields
\begin{equation}\label{AdCtrDmSbCn}
    \dom\breve{\partial}^2\delta_e(u,u^*)\subset-C(u,e,u^*).
\end{equation}
Moreover, since $\mU_{ad}(e)$ is polyhedric at every
$u\in\mU_{ad}(e)$ for any $u^*\in\mN(u,e)$, from \eqref{EqCmBnSbdf}
and \eqref{AdCtrCrCone} we deduce that
\begin{equation}\label{AdCtrEqCBcn}
    \breve{\partial}^2\delta_e(u,u^*)(v)=\breve{\partial}^2\delta(\cdot;\mU_{ad}(e))(u,u^*)(v)
    =\big(C(u,e,u^*)\big)^*,~\forall v\in-C(u,e,u^*).
\end{equation}

For every pair $(u,e)\in L^2(\Omega)\times E$ satisfying
$u\in\mU_{ad}(e)$, we define the three subsets of $\Omega$ as
follows
\begin{equation}\label{DjntSetsOm}
\begin{cases}
     \Omega_1(u,e)=\big\{x\in\Omega\st u(x)=\alpha(x)+e_\alpha(x)\big\},\\
     \Omega_2(u,e)=\big\{x\in\Omega\st u(x)\in\big(\alpha(x)+e_\alpha(x),\beta(x)+e_\beta(x)\big)\big\},\\
     \Omega_3(u,e)=\big\{x\in\Omega\st u(x)=\beta(x)+e_\beta(x)\big\}.
\end{cases}
\end{equation}
Then, $\mN(u,e)=N(u;\mU_{ad}(e))$ is computed by
\begin{equation}\label{VluesNrCne}
    \mN(u,e)=\left\{u^*\in L^2(\Omega)\left|
      \begin{aligned}
         &u^*(x)\leq0\quad\mbox{if}~x\in\Omega_1(u,e)\\
         &u^*(x)=0   \quad\mbox{if}~x\in\Omega_2(u,e)\\
         &u^*(x)\geq0\quad\mbox{if}~x\in\Omega_3(u,e)\\
         &\qquad for~a.e.~x\in\Omega
      \end{aligned}
\right.\right\}.
\end{equation}

Again, we obtain the equivalence of the conditions of the
abstract stability Theorems~\ref{ThmHdrCmBn} and \ref{ThmLpHdrEqv}.

\begin{Theorem}\label{ThmLpAdCtCh}
Assume that the assumptions {\rm\textbf{(A1)}-\textbf{(A3)}} hold.
Let $(\ou,\be)\in\dom\psi$ be such that there exists $\sigma>0$
satisfying
\begin{equation}\label{CndAdCtrSt}
    \alpha(x)+\be_\alpha(x)+\sigma\leq\beta(x)+\be_\beta(x)~for~a.e.~x\in\Omega,
\end{equation}
where $\be=(\be_y,\be_J,\be_\alpha,\be_\beta)$. Let any
$\ou^*\in\partial_u\psi(\ou,\be)$ be given. Then, the following
statements are equivalent:
\begin{itemize}
\item[{\rm(i)}] The control $\ou$ is a Lipschitzian fully stable local
minimizer of $\mP(\ou^*,\be)$ in \eqref{TprturbPro}.
\item[{\rm(ii)}] The control $\ou$ is a H\"{o}lderian fully stable local
minimizer of $\mP(\ou^*,\be)$ in \eqref{TprturbPro}.
\item[{\rm(iii)}] There are $\eta,\delta>0$ such that for all
$(u,e,u^*)\in\gph\partial_u\psi\cap\oB_\eta(\ou,\be,\ou^*)$ we have
\begin{equation}\label{HldrChrCB}
    \langle v^*,v\rangle\geq\delta\|v\|^2_{L^2(\Omega)}~\;\mbox{whenever}~\;
    v^*\in\breve{\partial}^2\psi_e(u,u^*)(v)~\mbox{for}~v\in L^2(\Omega).
\end{equation}
\end{itemize}
\end{Theorem}
\begin{proof}
For the function $\psi(u,e)$ in \eqref{FncPsi}, we have
$\partial_u\psi(u,e)=\mJ'_u(u,e)+\mN(u,e)$ whenever
$(u,e)\in\mU_{ad}(e)\times E$. Due to \eqref{CndAdCtrSt}, by
choosing $\delta\in(0,\sigma/2)$, we have $\mU_{ad}(e)\neq\emptyset$
for every $e\in\oB_\delta(\be)$. We first verify that the BCQ holds
at $(\ou,\be)\in\dom\psi$, i.e., the epigraphical mapping
$\mE(e):=\epi\psi(\cdot,e)$ is locally Lipschitz-like around the
point $(\be,\ou,\psi(\be,\ou))\in\gph\mE$. Select any
$e_1=(e^1_y,e^1_J,e^1_\alpha,e^1_\beta)\in\oB_\delta(\be)$ and
$e_2=(e^2_y,e^2_J,e^2_\alpha,e^2_\beta)\in\oB_\delta(\be)$ with
$e_1\neq e_2$ (because there is nothing to do if $e_1=e_2$). For any
$(u_1,\tau_1)\in \mE(e_1)\cap\oB_\delta\big(\ou,\psi(\ou,\be)\big)$,
we construct $u_2\in\mU_{ad}(e_2)$ as follows. Since
$u_1\in\mU_{ad}(e_1)$ with
$$\mU_{ad}(e_1)=\big\{u\in L^2(\Omega)\bst\alpha(x)+e^1_\alpha(x)
\leq u(x)\leq\beta(x)+e^1_\beta(x)~\mbox{for
a.e.}~x\in\Omega\big\},$$ we have
$u_1=\lambda(\alpha+e^1_\alpha)+(1-\lambda)(\beta+e^1_\beta)$ with
some $\lambda(x)\in[0,1]$ for a.e. $x\in\Omega$. By setting
\begin{equation}\label{DefUe2u2}
    u_2:=\lambda(\alpha+e^2_\alpha)+(1-\lambda)(\beta+e^2_\beta),
\end{equation}
we obtain $u_2\in\mU_{ad}(e_2)$ and
$$|u_1-u_2|\leq\lambda|e^1_\alpha-e^2_\alpha|+(1-\lambda)|e^1_\beta-e^2_\beta|
\leq|e^1_\alpha-e^2_\alpha|+|e^1_\beta-e^2_\beta|$$ for a.e.
$x\in\Omega$. This yields
$$\|u_1-u_2\|_{L^\infty(\Omega)}
  \leq\|e^1_\alpha-e^2_\alpha\|_{L^\infty(\Omega)}+\|e^1_\beta-e^2_\beta\|_{L^\infty(\Omega)}
  \leq\|e_1-e_2\|_E.$$
It follows that
$$\|u_1-u_2\|_{L^2(\Omega)}\leq|\Omega|^{1/2}\|e_1-e_2\|_E.$$
By choosing
$\tau_2=\tau_1+\ell_\psi\|(u_1,e_1)-(u_2,e_2)\|_{L^2(\Omega)\times
E}\geq\psi(u_2,e_2)$, where $\ell_\psi>0$ is a Lipschitz constant of
$\psi$ around $(\ou,\be)$, we have $(u_2,\tau_2)\in\mE(e_2)$ and
$$\begin{aligned}
    \|(u_1,\tau_1)-(u_2,\tau_2)\|_{L^2(\Omega)\times\R}
    &=\|u_1-u_2\|_{L^2(\Omega)}+|\tau_1-\tau_2|\\
    &=\|u_1-u_2\|_{L^2(\Omega)}+\ell_\psi\|(u_1,e_1)-(u_2,e_2)\|_{L^2(\Omega)\times E}\\
    &=(\ell_\psi+1)\|u_1-u_2\|_{L^2(\Omega)}+\ell_\psi\|e_1-e_2\|_E\\
    &\leq\big((\ell_\psi+1)|\Omega|^{1/2}+\ell_\psi\big)\|e_1-e_2\|_E.
\end{aligned}$$
Therefore, setting
$\ell_{\mE}=(\ell_\psi+1)|\Omega|^{1/2}+\ell_\psi>0$, we obtain
$$(u_1,\tau_1)\in(u_2,\tau_2)+\ell_{\mE}\|e_1-e_2\|_E\oB_{L^2(\Omega)\times\R}.$$
We have shown that
$$\mE(e_1)\cap V\subset\mE(e_2)+\ell_{\mE}\|e_1-e_2\|_E\oB_{L^2(\Omega)\times\R},~\forall e_1,e_2\in U,$$
where $U:=\oB_\delta(\be)$ and
$V:=\oB_\delta\big(\ou,\psi(\ou,\be)\big)$. Thus, $\mE$ is locally
Lipschitz-like around the point $(\be,\ou,\psi(\be,\ou))\in\gph\mE$.

We now check that $\psi$ is parametrically continuously prox-regular
at the point $(\ou,\be)$ for $\ou^*\in\partial_u\psi(\ou,\be)$. Let
$r>0$ satisfy $\|\mJ''_{uu}(u,e)\|_{L^2(\Omega)\times
L^2(\Omega)}\leq r$ for every
$(u,e)\in\mU_{ad}(e)\times\oB_\delta(\be)$. Fix an arbitrary
$\varepsilon>0$ and take any
$u^*\in\partial_u\psi(v,e)\cap\oB_\varepsilon(\ou)$ with
$(v,e)\in\oB_\varepsilon(\ou)\times\oB_\delta(\be)$ and
$\psi(v,e)\leq\psi(\ou,\be)+\varepsilon$. From this we have
$v\in\mU_{ad}(e)$. For every $u\in\oB_\varepsilon(\ou)$, if
$u\not\in\mU_{ad}(e)$, then we get $\psi(u,e)=+\infty$, and thus
\begin{equation}\label{InqCdPrRg}
    \psi(u,e)\geq\psi(v,e)+\langle u^*,u-v\rangle-\frac{r}{2}\|u-v\|^2_{L^2(\Omega)}.
\end{equation}
Consider the case that $u\in\mU_{ad}(e)$. Note that since
$u^*\in\partial_u\psi(v,e)$, we have $u^*=\mJ'_u(v,e)+u^*_v$ for
some $u^*_v\in\mN(v,e)$. Using a Taylor expansion, we have
\begin{equation}\label{IeqTylrPRg}
\begin{aligned}
    \mJ'_u(v,e)(u-v)
    &=\mJ(u,e)-\mJ(v,e)-\frac{1}{2}\mJ''_{uu}(\hu,e)(u-v)^2\\
    &\leq\mJ(u,e)-\mJ(v,e)+\frac{r}{2}\|u-v\|^2_{L^2(\Omega)},
\end{aligned}
\end{equation}
where $\hu=v+\theta(u-v)$ with $\theta\in(0,1)$. Since
$u^*=\mJ'_u(v,e)+u^*_v$ with $u^*_v\in\mN(v,e)$, from
\eqref{IeqTylrPRg} we get
$$\begin{aligned}
    \langle u^*,u-v\rangle\leq\mJ(u,e)-\mJ(v,e)+\frac{r}{2}\|u-v\|^2_{L^2(\Omega)}.
\end{aligned}$$
This yields that \eqref{InqCdPrRg} holds. In addition, the mapping
$(u,e,u^*)\mapsto\psi(u,e)$ is also continuous relative to
$\gph\partial_u\psi$ at $(\ou,\be,\ou^*)$. We have shown that $\psi$
is parametrically continuously prox-regular at $(\ou,\be)$ for
$\ou^*\in\partial_u\psi(\ou,\be)$.

Finally, we show that the graphical mapping
$F(e):=\gph\partial_u\psi(\cdot,e)$ is locally Lipschitz-like around
$(\be,\ou,\ou^*)\in\gph F$. For every
$e_1=(e^1_y,e^1_J,e^1_\alpha,e^1_\beta),e_2=(e^2_y,e^2_J,e^2_\alpha,e^2_\beta)\in\oB_\delta(\be)$,
pick any $(u_1,u^*_1)\in F(e_1)\cap\oB_\delta(\ou,\ou^*)$. Then, we
have $u^*_1=\mJ'_u(u_1,e_1)+\hu^*$ for some $\hu^*\in\mN(u_1,e_1)$
due to $\partial_u\psi(u,e)=\mJ'_u(u,e)+\mN(u,e)$. By constructing
$u_2$ the same as in \eqref{DefUe2u2}, we obtain
$u_2\in\mU_{ad}(e_2)$ and
\begin{equation}\label{EsU1U2L2}
    \|u_1-u_2\|_{L^2(\Omega)}\leq|\Omega|^{1/2}\|e_1-e_2\|_E.
\end{equation}
According to \eqref{DjntSetsOm}, from the definition of $u_2$ we get
$\Omega_i(u_1,e_1)=\Omega_i(u_2,e_2)$ for all $i=1,2,3$. By
\eqref{VluesNrCne}, it follows that $\mN(u_1,e_1)=\mN(u_2,e_2)$.
Thus, by setting
$$u^*_2:=\mJ'_u(u_2,e_2)+\hu^*\in\mJ'_u(u_2,e_2)+\mN(u_2,e_2)=\partial_u\psi(u_2,e_2),$$
we have $(u_2,u^*_2)\in F(e_2)$ and by \eqref{EsU1U2L2} we deduce
that
$$\begin{aligned}
    \|(u_1,u^*_1)-(u_2,u^*_2)\|_{L^2(\Omega)\times L^2(\Omega)}
    &=\|u_1-u_2\|_{L^2(\Omega)}+\|u^*_1-u^*_2\|_{L^2(\Omega)}\\
    &=\|u_1-u_2\|_{L^2(\Omega)}+\|\mJ'_u(u_1,e_1)-\mJ'_u(u_2,e_2)\|_{L^2(\Omega)}\\
    &\leq\|u_1-u_2\|_{L^2(\Omega)}+\ell_{\mJ'_u}\|(u_1,e_1)-(u_2,e_2)\|_{L^2(\Omega)\times E}\\
    &\leq(\ell_{\mJ'_u}+1)\|u_1-u_2\|_{L^2(\Omega)}+\ell_{\mJ'_u}\|e_1-e_2\|_E\\
    &\leq(\ell_{\mJ'_u}+1)|\Omega|^{1/2}\|e_1-e_2\|_E+\ell_{\mJ'_u}\|e_1-e_2\|_E\\
    &=\big((\ell_{\mJ'_u}+1)|\Omega|^{1/2}+\ell_{\mJ'_u}\big)\|e_1-e_2\|_E,
\end{aligned}$$
where $\ell_{\mJ'_u}>0$ is a Lipschitz constant of
$\mJ'_u(\cdot,\cdot)$ around the point $(\ou,\be)$. Therefore, by
setting $\ell_F:=(\ell_{\mJ'_u}+1)|\Omega|^{1/2}+\ell_{\mJ'_u}>0$,
we obtain
$$(u_1,u^*_1)\in(u_2,u^*_2)+\ell_F\|e_1-e_2\|_E\oB_{L^2(\Omega)\times L^2(\Omega)}.$$
This implies that
$$F(e_1)\cap V_0\subset F(e_2)+\ell_F\|e_1-e_2\|_E\oB_{L^2(\Omega)\times L^2(\Omega)},~\forall e_1,e_2\in U_0,$$
where $U_0:=\oB_\delta(\be)$ and $V_0:=\oB_\delta(\ou,\ou^*)$. This
means that $F$ is locally Lipschitz-like around the point
$(\be,\ou,\ou^*)\in\gph F$.

By Theorem~\ref{ThmHdrCmBn} and \cite[Corollary~4.8]{MorNgh14SIOPT},
the assertion of the theorem is fulfill. $\hfill\Box$
\end{proof}

\begin{Remark}\rm
Note that when the parametric space $E=L^2(\Omega)\times
L^2(\Omega)\times L^\infty(\Omega)\times L^\infty(\Omega)$ is
replaced with $\widetilde{E}=L^2(\Omega)\times L^2(\Omega)\times
L^2(\Omega)\times L^2(\Omega)$ in Theorem~\ref{ThmLpAdCtCh}, the
full stability properties cannot hold. Indeed, the set $\mU_{ad}(e)$
has no interior points in $L^2(\Omega)$, hence there are
perturbations $e$ arbitrarily close to $\be$ in $\widetilde{E}$ for
which such $\mU_{ad}(e)$ can be empty. This is the main reason why
the full stability properties fail to hold for this setting.
Therefore, we can investigate the full stability for
problem~\eqref{FxturbPro} in the corresponding perturbed
problem~\eqref{TprturbPro} with respect to the parametric metric
space
$\widetilde{E}_0=\{e\in\widetilde{E}\st\mU_{ad}(e)\neq\emptyset\}$
instead of $\widetilde{E}$. Note further that in our setting of
Theorem~\ref{ThmLpAdCtCh}, where the parametric space $E$ is
considered, condition~\eqref{CndAdCtrSt} ensures that
$\mU_{ad}(e)\neq\emptyset$ for $e$ near $\be$ enough.
\end{Remark}

\begin{Theorem}\label{ThmCFHldr}
Assume that the assumptions {\rm\textbf{(A1)}-\textbf{(A3)}} hold.
Let $(\ou,\be)\in\mU_{ad}(\be)\times E$ satisfy
condition~\eqref{CndAdCtrSt} for some $\sigma>0$, where
$\be=(\be_y,\be_J,\be_\alpha,\be_\beta)$. Let
$\ou^*\in\mJ'_u(\ou,\be)+\mN(\ou,\be)$ and define
$\hu^*=\ou^*-\mJ'_u(\ou,\be)\in\mN(\ou,\be)$. Then, the following
statements are equivalent:
\begin{itemize}
\item[{\rm(i)}] The control $\ou$ is a Lipschitzian fully stable local
minimizer of $\mP(\ou^*,\be)$ in \eqref{TprturbPro}.
\item[{\rm(ii)}] The control $\ou$ is a H\"{o}lderian fully stable
local minimizer of $\mP(\ou^*,\be)$ in \eqref{TprturbPro}.
\item[{\rm(iii)}] There exist $\eta>0$, $\delta>0$ such that for
each $(u,e,u^*)\in\gph\mN\cap\oB_\eta(\ou,\be,\hu^*)$, we have
\begin{equation}\label{CndChrHldr}
    \mJ''_{uu}(u,e)v^2\geq\delta\|v\|^2_{L^2(\Omega)},~\forall v\in C(u,e,u^*),
\end{equation}
where $C(u,e,u^*)$ is given by \eqref{AdCtrCrCone}.
\end{itemize}
\end{Theorem}
\begin{proof}
We observe that our assumptions satisfy all the assumptions of
Theorem~\ref{ThmLpAdCtCh}, hence by Theorem~\ref{ThmLpAdCtCh} we
deduce that (i) is equivalent to (ii).

We now verify that (ii) is equivalent to (iii). Suppose that the
control $\ou$ is a H\"{o}lderian fully stable local minimizer of the
problem $\mP(\ou^*,\be)$ in \eqref{TprturbPro}. According to
Theorem~\ref{ThmLpAdCtCh}, condition~\eqref{HldrChrCB} holds for
some $\eta_1,\delta_1>0$ and for all
$(u,e,u^*)\in\gph\partial_u\psi\cap\oB_{\eta_1}(\ou,\be,\ou^*)$.

Let us select any $v\in C(u,e,u^*)$ for
$(u,e,u^*)\in\gph\mN\cap\oB_\eta(\ou,\be,\hu^*)$ with some constant
$\eta>0$ satisfying $\eta_1/3\geq\eta(\ell+1)>0$, where $\ell>0$ is
a Lipschitz constant of $\mJ'_u(\cdot,\cdot)$ around the point
$(\ou,\be)$. Since
$\partial_u\psi(u,e)=\mJ'_u(u,e)+\mN(u,e)=(\mJ_e)'(u)+\mN_e(u)=\partial\psi_e(u)$,
we have $(\mJ_e)'(u)+u^*=\mJ'_u(u,e)+u^*\in\partial\psi_u(u,e)$. In
addition, we have
$$\begin{aligned}
   &\big\|\big(u,e,(\mJ_e)'(u)+u^*\big)-\big(\ou,\be,\ou^*\big)\big\|_{L^2(\Omega)\times E\times L^2(\Omega)}\\
   &\qquad =\|u-\ou\|_{L^2(\Omega)}+\|e-\be\|_E+\|(\mJ_e)'(u)+u^*-\ou^*\|_{L^2(\Omega)}\\
   &\qquad\leq\eta_1/3+\eta_1/3+\|(\mJ_e)'(u)-(\mJ_e)'(\ou)\|_{L^2(\Omega)}
          +\|(\mJ_e)'(\ou)+u^*-\ou^*\|_{L^2(\Omega)}\\
   &\qquad=2\eta_1/3+\|(\mJ_e)'(u)-(\mJ_e)'(\ou)\|_{L^2(\Omega)}+\|u^*-\hu^*\|_{L^2(\Omega)}\\
   &\qquad\leq2\eta_1/3+\ell\eta+\eta\leq\eta_1.
\end{aligned}$$
This implies that
$\big(u,e,(\mJ_e)'(u)+u^*\big)\in\gph\partial_u\psi\cap\oB_{\eta_1}(\ou,\be,\ou^*)$.
It follows from \eqref{AdCtrEqCBcn} with noting $u^*\in
C(u,e,u^*)^*$ that
\begin{equation}\label{IneqSOdrDer}
\begin{aligned}
    -(\mJ_e)''(u)v+u^*
    &\in(\mJ_e)''(u)(-v)+\breve{\partial}^2\delta_e(u,u^*)(-v)\\
    &=\big((\mJ_e)''(u)+\breve{\partial}^2\delta_e(u,u^*)\big)(-v)\\
    &=\big((\mJ_e)''(u)+\widehat{D}^*\mN_e(u,u^*)\big)(-v).\\
\end{aligned}
\end{equation}
According to \cite[Theorem~1.62]{Mor06Ba}, we have
\begin{equation}\label{SmRleDerv}
\begin{aligned}
    \breve{\partial}^2\psi_e\big(u,(\mJ_e)'(u)+u^*\big)(-v)
    &=\widehat{D}^*(\partial\psi_e)\big(u,(\mJ_e)'(u)+u^*\big)(-v)\\
    &=\widehat{D}^*\big((\mJ_e)'+\mN_e\big)\big(u,(\mJ_e)'(u)+u^*\big)(-v)\\
    &=\big((\mJ_e)''(u)+\widehat{D}^*\mN_e(u,u^*)\big)(-v).
\end{aligned}
\end{equation}
From \eqref{IneqSOdrDer} and \eqref{SmRleDerv} we get
\begin{equation}\label{InclHdrCmB}
    u^*-(\mJ_e)''(u)v\in\breve{\partial}^2\psi_e\big(u,(\mJ_e)'(u)+u^*\big)(-v).
\end{equation}
Since $v\in C(u,e,u^*)$, using condition~\eqref{HldrChrCB} it
follows from \eqref{InclHdrCmB} that
$$(\mJ_e)''(u)v^2=\langle u^*-(\mJ_e)''(u)v,-v\rangle\geq\delta_1\|-v\|^2_{L^2(\Omega)},$$
or, equivalently, as follows
$$\mJ''_{uu}(u,e)v^2\geq\delta_1\|v\|^2_{L^2(\Omega)}.$$
By choosing $\delta=\delta_1$, we obtain \eqref{CndChrHldr} in
(iii).

Conversely, suppose that (iii) holds. To prove (ii) it suffices to
verify that Theorem~\ref{ThmLpAdCtCh}(iii) holds for some
$\eta_1,\delta_1>0$. Given an arbitrary $v\in L^2(\Omega)$, pick any
$v^*\in\breve{\partial}^2\psi_e(u,u^*)(v)$ with
$(u,e,u^*)\in\gph\partial_u\psi\cap\oB_{\eta_1}(\ou,\be,\ou^*)$ with
$\eta/3\geq\eta_1(1+\ell)>0$. According to
\cite[Theorem~1.62]{Mor06Ba}, we have
$$\begin{aligned}
    \breve{\partial}^2\psi_e(u,u^*)(v)
    &=\widehat{D}^*(\partial\psi_e)(u,u^*)(v)\\
    &=\widehat{D}^*\big((\mJ_e)'+\mN_e\big)(u,u^*)(v)\\
    &=(\mJ_e)''(u)v+\breve{\partial}^2\delta_e\big(u,u^*-(\mJ_e)'(u)\big)(v).
\end{aligned}$$
This yields
$v^*-(\mJ_e)''(u)v\in\breve{\partial}^2\delta_e\big(u,u^*-(\mJ_e)'(u)\big)(v)$.
From \eqref{AdCtrDmSbCn} we deduce that
\begin{equation}\label{DKvInCone}
    v\in\dom\breve{\partial}^2\delta_e\big(u,u^*-(\mJ_e)'(u)\big)
    \subset-C\big(u,e,u^*-(\mJ_e)'(u)\big).
\end{equation}
Since
$v^*-(\mJ_e)''(u)v\in\breve{\partial}^2\delta_e\big(u,u^*-(\mJ_e)'(u)\big)(v)
=\widehat{D}^*\mN_e\big(u,u^*-(\mJ_e)'(u)\big)(v)$ and $\mN_e$ is a
maximal monotone operator, applying \cite[Lemma~A.2]{MorNgh13NA} we
get $\langle v^*-(\mJ_e)''(u)v,v\rangle\geq0$. We see that
$$\begin{aligned}
    &\big\|\big(u,e,u^*-(\mJ_e)'(u)\big)-(\ou,\be,\hu^*)\big\|_{L^2(\Omega)\times L^2(\Omega)}\\
    &\qquad=\|u-\ou\|_{L^2(\Omega)}+\|e-\be\|_E+\|u^*-(\mJ_e)'(u)-\hu^*\|_{L^2(\Omega)}\\
    &\qquad\leq\eta/3+\eta/3+\|u^*-\ou^*\|_{L^2(\Omega)}+\|(\mJ_e)'(u)-(\mJ_e)'(\ou)\|_{L^2(\Omega)}\\
    &\qquad\leq2\eta/3+\eta_1+\ell\eta_1\leq\eta.
\end{aligned}$$
Combining the latter with \eqref{DKvInCone} and using
condition~\eqref{CndChrHldr} in (iii) of the theorem we deduce that
$$(\mJ_e)''(u)v^2=\mJ''_{uu}(u,e)(-v,-v)\geq\delta\|-v\|^2_{L^2(\Omega)}=\delta\|v\|^2_{L^2(\Omega)}.$$
From this and the fact that $\langle
v^*-(\mJ_e)''(u)v,v\rangle\geq0$ we obtain the estimate
$$\langle v^*,v\rangle=\langle v^*-(\mJ_e)''(u)v,v\rangle+(\mJ_e)''(u)v^2
  \geq\delta_1\|v\|^2_{L^2(\Omega)},$$
where $\delta_1:=\delta$. We have shown that
Theorem~\ref{ThmLpAdCtCh}(iii) holds for the positive
$\eta_1,\delta_1$. $\hfill\Box$
\end{proof}

\medskip
In the sequel, we will equivalently characterize condition
\eqref{CndChrHldr} by an assumption on the reference point
$(\ou,\be)$ only.

We now define the sequential outer limits of the critical cones
$C(u,e,u^*)$ respectively via \eqref{OtrLimit} in the weak* topology
of $L^2(\Omega)$ by
\begin{equation}\label{LmWStopoAd}
    \mC_{w^*}(\ou,\be,\hu^*)=\Limsup_{(u,e,u^*)\stackrel{\gph\mN}\longrightarrow(\ou,\be,\hu^*)}C(u,e,u^*),
\end{equation}
and in the strong topology of $L^2(\Omega)$ as follows
\begin{equation}\label{LmSTtopoAd}
    \mC_s(\ou,\be,\hu^*)=\Big\{v\in L^2(\Omega)\Bst\exists(u_n,e_n,u^*_n)
    \stackrel{\gph\mN}\longrightarrow(\ou,\be,\hu^*),v_n\in C(u_n,e_n,u^*_n),v_n\to v\Big\}.
\end{equation}

\begin{Lemma}\label{LemOlmWS4}
Let any $(\ou,\be)\in\mU_{ad}(\be)\times E$ satisfy
\eqref{CndAdCtrSt} for some $\sigma>0$ and let
$\hu^*\in\mN(\ou,\be)$, where
$\be=(\be_y,\be_J,\be_\alpha,\be_\beta)$. Then,
$\mC_{w^*}(\ou,\be,\hu^*)$ and $\mC_s(\ou,\be,\hu^*)$ are computed
by the formula
\begin{equation}\label{CmpWSSTcAd}
    \mC_{w^*}(\ou,\be,\hu^*)=\mC_s(\ou,\be,\hu^*)
    =\big\{v\in L^2(\Omega)\bst v(x)\hu^*(x)=0\ \mbox{for a.e.}\ x\in\Omega\big\},
\end{equation}
where $\mC_{w^*}(\ou,\be,\hu^*)$ and $\mC_s(\ou,\be,\hu^*)$ are
respectively given by \eqref{LmWStopoAd} and \eqref{LmSTtopoAd}.
\end{Lemma}
\begin{proof}
Under our assumptions, for each $(u,e)\in\mU_{ad}(e)\times E$ near
$(\ou,\be)$ enough, the perturbed admissible control set
$\mU_{ad}(e)$ is convex and polyhedric by Remark~\ref{RmkUadPoly}.
Moreover, due to \cite[Lemma~4.11]{BaBoSi14TAMS},
$T_{\mU_{ad}(e)}(u)$ is computed by
\begin{equation}\label{RpreTUadE}
    T_{\mU_{ad}(e)}(u)
    =\left\{v\in L^2(\Omega)\left|\,
    \begin{aligned}
        &v(x)\geq0~~\mbox{for}~x\in\Omega_1(u,e)\\
        &v(x)\leq0~~\mbox{for}~x\in\Omega_3(u,e)
    \end{aligned}\right.\right\},
\end{equation}
where $\Omega_1(u,e)$ and $\Omega_3(u,e)$ are defined by
\eqref{DjntSetsOm}. In addition, for any $u^*\in\mN(u,e)$, we have
\begin{equation}\label{RpreCueuStr}
\begin{aligned}
    C(u,e,u^*)
    &=T_{\mU_{ad}(e)}(u)\cap\{u^*\}^\bot\\
    &=\big\{v\in T_{\mU_{ad}(e)}(u)\bst v(x)u^*(x)=0~\mbox{for a.e.}~x\in\Omega\big\}.
\end{aligned}
\end{equation}
In order to prove \eqref{CmpWSSTcAd}, we first verify the following
inclusion
\begin{equation}\label{WSbWSSTcAd}
    \mC_{w^*}(\ou,\be,\hu^*)\subset\big\{v\in L^2(\Omega)\bst v(x)\hu^*(x)=0~\mbox{for a.e.}~x\in\Omega\big\}.
\end{equation}
Pick any $v\in\mC_{w^*}(\ou,\be,\hu^*)$. Due to \eqref{LmWStopoAd},
we can find sequences $(u_n,e_n,u^*_n)\to(\ou,\be,\hu^*)$ with
$(u_n,e_n,u^*_n)\in\gph\mN$ and $v_n\stackrel{w}\rightharpoonup v$
with $v_n\in C(u_n,e_n,u^*_n)$ for every $n\in\N$. This implies by
\eqref{RpreCueuStr} that $v_n\in T_{\mU_{ad}(e_n)}(u_n)$ and
$v_n(x)u^*_n(x)=0$ for a.e. $x\in\Omega$ for every $n\in\N$. This
yields for any measurable set $\Theta\subset\Omega$ that
$$\int_\Theta v(x)\hu^*(x)dx=0,$$
that is $v(x)\hu^*(x)=0$ for a.e. $x\in\Omega$. Hence,
\eqref{WSbWSSTcAd} has been verified.

Since $\mC_s(\ou,\be,\hu^*)\subset\mC_{w^*}(\ou,\be,\hu^*)$, to
obtain \eqref{CmpWSSTcAd} it suffices to prove that
\begin{equation}\label{SSpWSSTcAd}
    \big\{v\in L^2(\Omega)\bst v(x)\hu^*(x)=0~\mbox{for a.e.}~x\in\Omega\big\}\subset\mC_s(\ou,\be,\hu^*).
\end{equation}
Select any $v$ from the set on the left-hand side of
\eqref{SSpWSSTcAd}. We define functions $v_1$ and $v_2$ by
$$v_1(x)=\begin{cases}
    \max\{0,-v(x)\},  &\mbox{for}~x\in\Omega_1(\ou,\be)\\
    0,                &\mbox{for}~x\in\Omega_2(\ou,\be)\\
    \min\{0,-v(x)\},  &\mbox{for}~x\in\Omega_3(\ou,\be),\\
\end{cases}$$
and $v_2=v+v_1$. According to \eqref{RpreTUadE}, we have $v_1,v_2\in
T_{\mU_{ad}(\be)}(\ou)$. In addition, we also have
$v_1,v_2\in\{\hu^*\}^\bot$ due to $v(x)\hu^*(x)=0$ for a.e.
$x\in\Omega$. Thus, we get $v_1,v_2\in C(\ou,\be,\hu^*)$. Since
$\mU_{ad}(\be)$ is polyhedric, by \eqref{CndPlyhrc} we find
sequences $v_{1,n}\to v_1$, $v_{2,n}\to v_2$, $t_{1,n}\downarrow0$,
$t_{2,n}\downarrow0$ such that $\ou+t_{1,n}v_{1,n}\in\mU_{ad}(\be)$,
$\ou+t_{2,n}v_{2,n}\in\mU_{ad}(\be)$, and
$v_{1,n},v_{2,n}\in\{\hu^*\}^\bot$. Since $\mU_{ad}(\be)$ is convex,
by setting $t_n=\min\{t_{1,n},t_{2,n}\}$ we have
$$\begin{cases}
    u_n:=\ou+t_nv_{1,n}\in\mU_{ad}(\be)\\
    \wu_n:=\ou+t_nv_{2,n}\in\mU_{ad}(\be).
\end{cases}$$
Now, by choosing $e_n=\be$, $u^*_n=\hu^*$, and
$v_n=v_{2,n}-v_{1,n}$, we have $(u_n,e_n,u^*_n)\to(\ou,\be,\hu^*)$
with $(u_n,e_n,u^*_n)\in\gph\mN$ and $v_n\to v$ with
$$v_n=v_{2,n}-v_{1,n}=\frac{\wu_n-u_n}{t_n}\in\cone\big(\mU_{ad}(e_n)-u_n\big)\cap\{u^*_n\}^\bot
  \subset C(u_n,e_n,u^*_n).$$
This yields $v\in\mC_s(\ou,\be,\hu^*)$, and therefore
\eqref{SSpWSSTcAd} holds. $\hfill\Box$
\end{proof}

\medskip
From Theorem~\ref{ThmCFHldr} and Lemma~\ref{LemOlmWS4} we obtain the
following result.

\begin{Theorem}\label{ThmChLpHdAd}
Assume that the assumptions {\rm\textbf{(A1)}-\textbf{(A3)}} hold.
Let $(\ou,\be)\in\mU_{ad}(\be)\times E$ satisfy
condition~\eqref{CndAdCtrSt} for some $\sigma>0$, where
$\be=(\be_y,\be_J,\be_\alpha,\be_\beta)$. Let
$$\ou^*\in\zeta(\ou+\be_y)+\varphi_{\ou+\be_y}+G'(\ou+\be_y)^*\be_J+\mN(\ou,\be),$$
and define
$$\hu^*=\ou^*-\zeta(\ou+\be_y)-\varphi_{\ou+\be_y}-G'(\ou+\be_y)^*\be_J.$$
Then, the following statements are equivalent:
\begin{itemize}
\item[{\rm(i)}] The control $\ou$ is a Lipschitzian fully stable local
minimizer for $\mP(\ou^*,\be)$ in \eqref{TprturbPro}.
\item[{\rm(ii)}] The control $\ou$ is a H\"{o}lderian fully stable local
minimizer for $\mP(\ou^*,\be)$ in \eqref{TprturbPro}.
\item[{\rm(iii)}] The following condition holds that
\begin{equation}\label{CdChrFLpHdr}
    \mJ''_{uu}(\ou,\be)v^2>0,\ \forall v\neq0\ \mbox{with}\ v(x)\hu^*(x)=0~\mbox{for a.e.}~x\in\Omega,
\end{equation}
where $\be=(\be_y,\be_J,\be_\alpha,\be_\beta)$.
\end{itemize}
\end{Theorem}
\begin{proof}
As it has been seen in the proof of Theorem~\ref{ThmChLpBrUE}, we
have
$$\mJ'_u(\ou,\be)
  =\zeta(\ou+\be_y)+\varphi_{\ou+\be_y}+G'(\ou+\be_y)^*\be_J\in L^2(\Omega).$$
Therefore, the inclusions $\ou^*\in\mJ'_u(\ou,\be)+\mN(\ou,\be)$ and
$\hu^*=\ou^*-\mJ'_u(\ou,\be)\in\mN(\ou,\be)$ follow. By
Theorem~\ref{ThmCFHldr}, (i) is equivalent to (ii). In order to
verify the equivalence of (ii) and (iii), it suffices to prove that
(iii) is equivalent to Theorem~\ref{ThmCFHldr}(iii).

Assume that Theorem~\ref{ThmCFHldr}(iii) holds. Let any $v\in
L^2(\Omega)$ satisfy $v\neq0$ and $v(x)\hu^*(x)=0$ for a.e.
$x\in\Omega$. By \eqref{CmpWSSTcAd}, we have
$v\in\mC_s(\ou,\be,\hu^*)$. From this and \eqref{LmSTtopoAd} there
exist sequences $(u_n,e_n,u^*_n)\to(\ou,\be,\hu^*)$ with
$(u_n,e_n,u^*_n)\in\gph\mN$, $v_n\in C(u_n,e_n,u^*_n)$ such that
$v_n\to v$ when $n\to\infty$. By $v_n\in C(u_n,e_n,u^*_n)$ and by
\eqref{CndChrHldr} we get
\begin{equation}\label{SqCdChrHdr}
    \mJ''_{uu}(u_n,e_n)v^2_n\geq\delta\|v_n\|^2_{L^2(\Omega)},
\end{equation}
Passing \eqref{SqCdChrHdr} to the limit as $n\to\infty$, we obtain
$\mJ''_{uu}(\ou,\be)v^2>0$.

Conversely, assume that (iii) holds. Suppose to the contrary that
Theorem~\ref{ThmCFHldr}(iii) does not hold. Then, there exist
sequences $(u_n,e_n,u^*_n)\to(\ou,\be,\hu^*)$ with
$(u_n,e_n,u^*_n)\in\gph\mN$ and $v_n\in C(u_n,e_n,u^*_n)$ such that
\begin{equation}\label{CtrCndHdr}
    \mJ''_{uu}(u_n,e_n)v^2_n<\frac{1}{n}\|v_n\|^2_{L^2(\Omega)},\ \forall n\in\N.
\end{equation}
We may assume that $\|v_n\|_{L^2(\Omega)}=1$ for every $n\in\N$ and
may also assume that $v_n\rightharpoonup v$ in $L^2(\Omega)$. This
implies that $v\in\mC_{w^*}(\ou,\be,\hu^*)$ due to
\eqref{LmWStopoAd}. Using \eqref{CmpWSSTcAd} and arguing similarly
as in the proof of Theorem~\ref{ThmChLpBrUE}, we obtain a
contradiction. $\hfill\Box$
\end{proof}

\begin{Remark}\rm
We observe that under linear perturbations of the admissible control
set of the control problem, we obtain condition \eqref{CdChrFLpHdr}
with the same structure as condition \eqref{ExCdChrFLp}. In other
words, when the admissible control sets undergo linear
perturbations, condition \eqref{ExCdChrFLp}, which is a
characterization of both Lipschitzian and H\"{o}lderian full
stability, is still ``stable" provided that the assumption
\eqref{CndAdCtrSt} holds. Our result considerably extends the
results of \cite{MalTro99ZAA} and \cite{MalTro00CC} because we
consider basic perturbations of all the cost functional, the state
equation, and the admissible control set of the control problem
while the admissible control sets considered in \cite{MalTro99ZAA}
and \cite{MalTro00CC} are fixed.
\end{Remark}

\section{Concluding remarks}

Explicit characterizations of both Lipschitzian and H\"{o}lderian
full stability for a class of optimal control problems are
established in this paper. When the admissible control set of the
problems is fixed, we show that the two full stability properties
are always equivalent. This fact relies on a general stability
result for optimization problems. In the perturbed admissible
controls setting, the two full stability properties are also
equivalent for this class of control problems. From our results we
see that the equivalence of the two full stability properties in
this case is due to the special structures of the perturbed
admissible control sets, these full stability properties are not
always equivalent in general. We think that the main reason for the
equivalence of the full Lipschitzian stability and the full
H\"{o}lderian one in our paper is that all the cost functional, the
state equation, and especially the admissible control set of the
control problems undergo \emph{linear perturbations}. Motivated by
this remark, in order to understand deeply about phenomena related
to the full stability in optimal control, in future research we
intend to study the full stability for classes of control problems
under \emph{nonlinear perturbations}.

\end{document}